\documentclass[11pt]{amsart}
\usepackage{amsmath,amssymb,amsfonts}
\usepackage[mathscr]{eucal}

\usepackage{hyperref}
\hypersetup{
    colorlinks=true,
    linkcolor=blue,
    citecolor=red
}
\usepackage{cleveref}

\usepackage[margin=1.6in]{geometry}

\usepackage{enumerate}
\newenvironment{enumroman}{\begin{enumerate}[\upshape (i)]}
                                                {\end{enumerate}}
\usepackage[textwidth=3cm, textsize=small, colorinlistoftodos]{todonotes}

\usepackage{tikz}
\usetikzlibrary{decorations.pathreplacing,angles,quotes}
\usetikzlibrary{matrix, arrows}

\theoremstyle{plain}
\newtheorem{theorem}[equation]{Theorem}

\newtheorem{corollary}[equation]{Corollary}

\newtheorem{prop}[equation]{Proposition}
\newtheorem{proposition}[equation]{Proposition}
\newtheorem{lemma}[equation]{Lemma}

\newtheorem{convention}[equation]{Convention}

\theoremstyle{definition}
\newtheorem{definition}[equation]{Definition}

\newtheorem{example}[equation]{Example}
\newtheorem{remark}[equation]{Remark}
\newtheorem{remarks}[equation]{Remarks}

\newtheorem*{thank}{Acknowledgments}

\numberwithin{equation}{section}

\input xy
\xyoption{all}

\def \A {\mathcal {A}}
\def \B {\mathcal {B}}
\def \C {\mathcal {C}}
\def \D {\mathcal {D}}
\def \I {\mathcal {I}}

\def \M {\mathcal {M}}
\def \P {\mathcal{P}}
\def \s {\mathcal{S}}
\def \V {\mathcal {V}}
\def \W {\mathcal {W}}

\newcommand{\Ab}{\mathcal{A}b}
\newcommand{\colim}{\operatornamewithlimits{colim}}

\newcommand{\Deltaop}{\Delta^{\op}}
\newcommand{\Fun}{\operatorname{Fun}}

\newcommand{\hocofiber}{\operatorname{hocofiber}}
\newcommand{\hocolim}{\operatornamewithlimits{hocolim}}
\newcommand{\holim}{\operatornamewithlimits{holim}}
\newcommand{\Hom}{\operatorname{Hom}}

\newcommand{\Id}{\operatorname{id}}
\newcommand{\Map}{\operatorname{Map}}
\newcommand{\ob}{\operatorname{ob}}

\newcommand{\op}{\operatorname{op}}
\newcommand{\Sets}{\mathcal{S}ets}

\newcommand{\Top}{\mathcal Top}

\title{Enriched functor categories for functor calculus}

\author[Bandklayder]{Lauren Bandklayder}
\address{Max Planck Institute for Mathematics, Bonn, Germany}
\email{lbandklayder@gmail.com}

\author[Bergner]{Julia E. Bergner}
\address{Department of Mathematics, University of Virginia, Charlottesville, VA, USA}
\email{jeb2md@virginia.edu}

\author[Griffiths]{Rhiannon Griffiths}
\address{Department of Mathematics, Cornell University, Ithaca, NY, USA}
\email{rhiannon.griffiths@cornell.edu}

\author[Johnson]{Brenda Johnson}
\address{Department of Mathematics, Union College, Schenectady, NY, USA}
\email{johnsonb@union.edu}

\author[Santhanam]{Rekha Santhanam}
\address{Department of Mathematics, IIT\ Bombay, Powai, Mumbai, India}
\email{reksan@iitb.ac.in}

\thanks{This work was done as part of the Women in Topology Workshop in August 2019, supported by the Hausdorff Research Institute for Mathematics, NSF grant DMS 1901795, the AWM ADVANCE grant NSF-HRD-1500481, and Foundation Compositio Mathematica. LB was supported by the Max Planck Institute for Mathematics in Bonn. JB was partially supported by NSF grants DMS-1659931 and DMS-1906281. BJ was partially supported by the Union College Faculty Research Fund.} 



\begin{document}

\begin{abstract}
    In this paper we present background results in enriched category theory and enriched model category theory necessary for developing model categories of enriched functors suitable for doing functor calculus.  
\end{abstract}

\maketitle

\tableofcontents

\section{Introduction}

Functor calculi are important tools in algebraic topology that have been used to produce significant results in a wide range of  fields. For example, the homotopy functor calculus of Goodwillie \cite{goodwillie3} has had applications in algebraic $K$-theory \cite{mccarthy} and $v_n$-periodic homotopy theory \cite{am}, \cite{heuts}.  The orthogonal calculus of Goodwillie and Weiss \cite{gw}, \cite{weissorth} and manifold calculus of Weiss \cite{weissemb} have been used to study embeddings and immersions of manifolds, characteristic classes, and spaces of knots; for example, see \cite{ltv}, \cite{rw}, and \cite{sinha}.  

Each functor calculus provides a means of approximating a functor $F$ between particular types of categories with a tower of functors under $F$
\[ \xymatrix{&&&F\ar[llld]\ar[lld]\ar[ld]\ar[d]\ar[dr]\ar[drr]&&\\P_0F&P_1F\ar[l]&\dots\ar[l]&P_nF\ar[l]&P_{n+1}F\ar[l]&\dots\ar[l]} \]
that is analogous to a  Taylor series for a function, with each $P_nF$ being in some sense a degree $n$ approximation to $F$. Such towers are referred to as \emph{Taylor towers}.  The functor calculi also provide means to classify ``homogeneous degree $n$" functors (degree $n$ functors with trivial degree $n-1$ approximations), which arise as the layers of the towers (homotopy fibers between consecutive terms) in terms of more tractable objects. Because of these classification results, it is often easier to identify the layers in a Taylor tower than it is to identify the $P_nF$'s.  By building appropriate model structures on functor categories, these polynomial approximations can be interpreted as fibrant replacements for functors.  This point of view has been developed by Barnes and Oman \cite{bo} in the case of the orthogonal calculus, and by Biedermann, Chorny, and R\"ondigs \cite{bcr}, \cite{br} and by Barnes and Eldred \cite{be1} in the case of the homotopy functor calculus.  This perspective makes it possible to upgrade the classifications of homogeneous degree $n$ functors from equivalences of homotopy categories to Quillen equivalences between model categories.  This perspective has also led to means by which different functor calculi can be extended to new contexts \cite{taggart} and compared to one another \cite{be2}, \cite{taggartcomp}; see also \cite{lurie} and \cite{pereira}.

This paper grew from a desire to apply this model category-theoretic approach to the discrete functor calculus of Bauer, McCarthy, and the fourth-named author \cite{bjm}. Such model structures have been identified for abelian versions of this type of calculus by Renaudin \cite{renaudin} and Richter \cite{richter}.  Using similar techniques, it is straightforward to establish the existence of such model structures for the discrete calculus, but we are interested in developing these structures in a simplicially enriched context.  

Our motivation for doing so is inspired by the work of Biedermann and R\"ondigs \cite{br}.  They develop a simplicially enriched version of Goodwillie's homotopy functor calculus in such a way that their model structure for $n$-excisive functors is cofibrantly generated.  Because this additional structure is quite powerful, as we develop a ``degree $n$" model structure for discrete functor calculus, we want to employ a similar strategy.

However, following this strategy requires a good understanding of enriched functor categories, and many of the proofs of the results we need can be difficult to find explicitly in the literature.  In this paper, we aim to bring together these results in a relatively self-contained treatment, with an eye toward recognizing the common features between these two kinds of functor calculus.  This paper can thus be regarded as the preparatory work to developing the cofibrantly generated model structures that will be the main results of the sequel \cite{sequel}.  

As an application of these ideas, and as further preparation for that work, we show that one of the basic building blocks for discrete functor calculus, the comonad $\bot_n$, has an isomorphic simplicial representable version $\bot_n^R$, and that the existence of this version and a corresponding construction for the Goodwillie calculus developed in \cite{br} are both consequences of a more general result concerning a construction we refer to as the \emph{evaluated cotensor}.  We show, that although $\V$-enriched functor categories $\Fun(\C,\D)$ are not generally enriched in $\Fun(\C,\V)$, the evaluated cotensor behaves quite like an ordinary cotensor. The category $\Fun(\C,\D)$ therefore enjoys much of the structure of a category enriched, tensored, and contensored in $\Fun(\C,\V)$.  

For each flavor of functor calculus, both the original definition and its variant in terms of representables have distinct advantages; however, the latter is more amenable to working in an enriched setting, as is our goal here.   For the functor $T_n$ from homotopy functor calculus and the functor $\bot_n$ from discrete functor calculus, we show that the two definitions agree and, as the main result of the last section of this paper, prove that the representable variants of the functors $T_n$ and $\bot_n$ are both simplicial functors.   With this structure in place, we are well-positioned to recover the cofibrantly generated $n$-excisive model structure on Goodwillie's functor calculus from \cite{br}, and establish similar cofibrantly generated model structures for degree $n$ functors in the context of a broader class of functor calculi, including discrete calculus in \cite{sequel}.

\subsection*{Organization of the paper}

In Section \ref{sectionenriched} we provide fundamental background material on enriched categories.  We recall the definition of a $\V$-category, or category enriched in $\V$, and describe the additional structures of a $\V$-category being tensored or cotensored in $\V$.

In Section \ref{sectionyoneda} we provide an enriched version of the classical Yoneda Lemma and its dual.  This section includes the definitions of representable functors and ends and coends of certain bifunctors in the context of enriched categories.  Aside from being necessary for the Yoneda Lemma, these constructions are used repeatedly throughout this paper and its sequel. We also generalize the definition of an end to construct the ``evaluated cotensor,'' which is a bifunctor $\Fun(\C,\V)^{\op} \times \Fun(\C,\D) \rightarrow \D$ described by Biedermann and R\"ondigs in \cite{br}.

Motivated by applications to functor calculus, we specialize to enriched functor categories in Section \ref{sectionfunctors}. We show that, given $\V$-categories $\C$ and $\D$ with $\C$ small, the category $\Fun(\C,\D)$ of $\V$-functors from $\C$ to $\D$ can be viewed as a category enriched in $\V$.  We also give sufficient conditions for when $\Fun(\C, \D)$ is tensored or cotensored over $\V$ and establish some properties of the evaluated cotensor.

In Section \ref{sectionmodel}, we consider $\V$-categories with the additional structure of a $\V$-model category.  In particular, we give conditions on $\V$-categories $\C$ and $\D$ under which $\Fun(\C,\D)$ with the projective model structure is a $\V$-model category.

Starting in Section \ref{sectioncotensor}, we restrict to the case where $\V=\s$, the category of simplicial sets, and so work in the simplicial context.  We recall the formal definitions of homotopy limits and colimits in a simplicial model category, and use them to show how the evaluated cotensor interacts with these constructions.

We conclude in Section \ref{sectionfunctorcalc} with an application of these enriched category techniques to functor calculus. We begin by revisiting the construction of a version of Goodwillie's functor $T_n$ in terms of representable functors, as developed by Biedermann and R\"ondigs in \cite{br}, in Section \ref{tnrep}.  In Section \ref{discreterep}, we then develop the analogous construction for the functor $\bot_n$ that plays a similar role for discrete functor calculus.  Building on the similarities of the two constructions, in Section \ref{tnsimp} we prove that both $T_n$ and $\bot_n$ define simplicial functors.

\begin{thank} 
This paper was written as part of the Women in Topology Workshop in August 2019.  The authors would like to thank Ang\'elica Osorno and Sarah Whitehouse for organizing this workshop, as well as the Hausdorff Institute for its hospitality.  We would also like to thank David Barnes, Georg Biedermann, Nick Kuhn, Lyne Moser, Birgit Richter, and Stefan Schwede for conversations related to this project.
\end{thank}

\section{Enriched categories} \label{sectionenriched} 

In this section, we give some background results on categories enriched in a suitable monoidal category. These definitions are standard and can be found in \cite{kelly}, \cite{riehl2} and \cite{hovey}. 

\begin{definition} \cite[\S 1.1]{kelly} 
A \emph{monoidal category} $(\V, \otimes, I)$ is a category $\V$ equipped with a bifunctor $- \otimes - \colon \mathcal V \times \V \rightarrow \V$ sending a pair of objects $(V,W)$ to an object $V \otimes W$, and a unit object $I$, satisfying appropriate associativity and unit axioms.
\end{definition}

Given monoidal categories $\mathcal V$ and $\W$, one can consider \emph{monoidal functors} that preserve the monoidal structure \cite[4.1.2]{hovey}, and likewise \emph{monoidal natural transformations} \cite[4.1.3]{hovey}.

\begin{definition} \cite[\S 1.4]{kelly} 
A monoidal category $\V$ is \emph{symmetric} if, for any objects $V$ and $W$ of $\V$, there is an isomorphism $t \colon V \otimes W \rightarrow W \otimes V$ such that $t^2=\Id$ and various compatibility axioms are satisfied.
\end{definition}

\begin{definition}\cite[\S 1.5]{kelly}  
A symmetric monoidal category $\V$ is \emph{closed} if it is equipped with a bifunctor $(-)^{(-)} \colon \V \times \V^{\op} \rightarrow \V$ sending a pair of objects $(W,V)$ to $W^V$, such that given objects $U$, $V$, and $W$, there is a natural isomorphism
\[ \Hom_\V(U, W^V) \cong \Hom_\V(U \otimes V, W). \]
\end{definition}

The object $W^V$ is sometimes referred to as an \emph{internal hom object}, since we think of it as a mapping object which itself lives in the category $\V$.

\begin{example}  
The category $\Ab$ of abelian groups with the usual tensor product is a closed symmetric monoidal category in which the internal hom object $B^A$ between two abelian groups $A$ and $B$ is taken to be the abelian group of homomorphisms between $A$ and $B$. 
\end{example}

A primary example in this paper is the category of simplicial sets.

\begin{example} \label{ssetissymmon} 
Let $\Delta$ denote the standard simplex category, whose objects are finite ordered sets $[n]=\{0 \leq \cdots \leq n\}$ and whose morphisms are given by order-preserving maps of sets. The category $\s$ of simplicial sets, or functors $\Deltaop \rightarrow \Sets$, with the usual cartesian product is a closed symmetric monoidal category, where the internal hom object between two simplicial sets $U$ and $V$ is  the simplicial set $V^U$ whose $n$-simplices are given by 
\[ (V^U)_n:=\Hom_{\s}(U\times \Delta[n],V). \]
Here, recall that $\Delta[n]$ is the standard $n$-simplex, given by the representable functor $\Hom_{\Delta}(-,[n])$.
\end{example}

The following definition is a generalization of a closed monoidal category, in that we consider other categories $\D$ with hom objects in $\V$.

\begin{definition} \label{closedsymmon} \cite[\S 1.2]{kelly} 
Let $(\V, \otimes, I)$ be a monoidal category. A category $\D$ \emph{enriched in} $\V$, or a $\V$-\emph{category}, consists of a collection $\ob(\D)$ of objects of $\D$, together with an object $\Map_\D(X,Y)$ of $\V$ for every pair $X, Y \in \ob(\D)$, such that for each $X \in \ob(\D)$ there is an identity morphism
\[ i_X \colon I \rightarrow \Map_\D(X,X) \] 
in $\V$, and for each triple $X,Y,Z \in \ob(\D)$ there is a composition morphism 
\[ \circ_{XYZ} \colon \Map_\D(X,Y) \otimes \Map_\D(Y,Z) \rightarrow \Map_\D(X,Z) \]
in $\V$ satisfying appropriate associativity and unit axioms.
\end{definition}

We sometimes denote the mapping object $\Map_\D(X,Y)$ by $\Map(X,Y)$ when there is no ambiguity about the ambient category $\D$.

It is often helpful to distinguish between a $\V$-enriched category and its underlying ordinary category, which we now define.

\begin{definition}\label{underlyingcat}
The \emph{underlying category} of a $\V$-category $\D$ is the category $\D_0$ with the same objects as $\D$, $\Hom_{\D_0}(X,Y) = \Hom_{\V} \left(I, \Map_{\D}(X,Y) \right)$ and composition and identities induced by the composition and identity morphisms in $\D$.
\end{definition}

In particular, when we say that an ordinary category $\C$ is a $\V$-category, we mean that there exists a $\V$-category for which the underlying category is $\C$.  

When $\V$ is a closed monoidal category, the category $\V$ is itself enriched in $\V$, via the internal hom object. Given any objects $V, W$ of $\V$, we define
\[ \Map_\V(V,W) = W^V. \] 
In other words, a closed monoidal category is a category enriched in itself. 

\begin{remark}\label{mapi}
Observe that given a closed monoidal category $\V$ and an object $V$ in $\V$, the identity map $V \rightarrow V$ corresponds to a unique map $i \colon I \rightarrow V^V$ via the isomorphisms
\[ \Map_\V(V,V) \cong \Map_\V(I \otimes V, V) \cong \Map_\V(I, V^V). \]
\end{remark} 

Given two $\V$-categories, we are typically interested in functors between them that behave nicely with respect to their enrichments.  
\begin{definition} \cite[\S 1.2]{kelly}
Let $\C$ and $\D$ be $\V$-categories.  A \emph{$\V$-functor} $F \colon \C \rightarrow \D$ is a function assigning each object $A$ of $\C$ to an object $FA$ of $\D$, together with a morphism 
\[F_{AB} \colon \Map_\C(A,B) \rightarrow \Map_\D(FA,FB) \]
in $\V$ for each pair of objects $A$ and $B$ of $\C$, such that the collection of all such morphisms preserves composition and identity morphisms.
\end{definition} 

We can analogously define a $\V$-natural transformation between $\V$-functors.

\begin{definition}\label{V-natural transformations} \cite[\S 1.2]{kelly}
Let $\C$ and $\D$ be $\V$-categories and $F, G \colon \C \rightarrow \D$ be two $\V$-functors.  A $\V$-natural transformation  $\eta \colon F \Rightarrow G$ is a collection
\[ \{ \eta_A \colon I \rightarrow \Map_{\D}(FA,GA) \} \]
of morphisms in $\V$ where $A$ ranges over all objects in $\C$, such that the following diagram commutes for any objects $A$ and $B$ of $\C$:
\[ \begin{tikzpicture}[node distance=5cm, auto, x=.5mm, y=.8mm]

\node (A) {\small $\Map_{\C}(A,B)$};
\node (B) [node distance=4cm, right of=A] {\small $I \otimes \Map_{\C}(A,B)$};
\node (C) [node distance=7cm, right of=B] {\small $\Map_{\D}(FA,GA) \otimes \Map_{\D}(GA,GB)$};
\node (X) [node distance=2.5cm, below of=A] {\small $\Map_{\C}(A,B) \otimes I$};
\node (Y) [node distance=6cm, right of=X] {\small $\Map_{\D}(FA,FB) \otimes \Map_{\D}(FB,GB)$};
\node (Z) [right of=Y] {\small $\Map_{\mathcal{D}}(FA,GB)$.};

\draw[->] (A) to node {} (B);
\draw[->] (B) to node {\small $\eta_A \otimes G_{AB}$} (C);
\draw[->] (X) to node [swap] {\small $F_{AB} \otimes \eta_B$} (Y);
\draw[->] (Y) to node [swap] {\small $\circ$} (Z);
\draw[->] (A) to node [swap] {} (X);
\draw[->] (C) to node {\small $\circ$} (Z);

\end{tikzpicture} \]
\end{definition}

In particularly nice cases, a $\V$-enriched category interacts nicely with the monoidal category $\V$ via tensor and cotensor functors, generalizing features of the closed monoidal structure on $\V$.  We begin with the notion of a tensor functor, which associates to an object of $\D$ and an object of $\V$ another object of $\D$.

\begin{definition} \label{tensor def} \cite[3.7.2]{riehl2}
A $\V$-category $\D$ is \emph{tensored over} $\V$ if there exists a $\V$-functor 
\[ \begin{tikzpicture}[node distance=4cm, auto]
\node (A) {$\D \times \V$};
\node (B) [right of=A] {$\D$};
\node (X) [node distance=1cm, below of=A] {$(D,V)$};
\node (Y) [right of=X] {$D \otimes V$};
\node (X') [node distance=2cm, below of=X] {$(E,W)$};
\node (Y') [right of=X'] {$E \otimes W$};
\node (x) [node distance=1cm, below of=X] {};
\node (x') [node distance=0.5cm, right of=x] {};
\node (y) [node distance=1cm, below of=Y] {};
\node (y') [node distance=0.5cm, left of=y] {};
\draw[->] (A) to node {$- \otimes - $} (B);
\draw[->] (X) to node [swap] {$\Map_{\D}(D,E) \otimes W^V $} (X');
\draw[->] (Y) to node {$\Map_{\D}(D \otimes V, E \otimes W)$} (Y');
\draw[|->, dashed] (x') to node {} (y');
\end{tikzpicture} \]
together with a $\V$-natural isomorphism $\Map_{\D}(D \otimes V, E) \cong \Map_{\D}(D, E)^V$ in $\V$.
\end{definition}

Analogously, a cotensor functor also associates to an object of $\D$ and an object of $\V$ an object of $\D$, but in a manner more closely related to the internal hom object of $\V$.

\begin{definition} \label{cotensor def}  \cite[3.7.3]{riehl2}
A $\V$-category $\D$ is \emph{cotensored over} $\V$ if there exists a $\V$-functor 
\[  \begin{tikzpicture}[node distance=4cm, auto]

\node (A) {$\D \times \V^{\op}$};
\node (B) [right of=A] {$\D$};

\node (X) [node distance=1cm, below of=A] {$(D,V)$};
\node (Y) [right of=X] {$D^V$};
\node (X') [node distance=2cm, below of=X] {$(E,W)$};
\node (Y') [right of=X'] {$E^W$};
\node (x) [node distance=1cm, below of=X] {};
\node (x') [node distance=0.5cm, right of=x] {};
\node (y) [node distance=1cm, below of=Y] {};
\node (y') [node distance=0.5cm, left of=y] {};

\draw[->] (A) to node {$(-)^{(-)}$} (B);
\draw[->] (X) to node [swap] {$\Map_{\D}(D,E) \otimes V^W$} (X');
\draw[->] (Y) to node {$\Map_{\D}(D^V, E^W)$} (Y');
\draw[|->, dashed] (x') to node {} (y');
\end{tikzpicture} \]
together with a $\V$-natural isomorphism $\Map_{\D}(D, E^V) \cong \Map_{\D}(D, E)^V$.
\end{definition}

We can use the precise formulations of these definitions to show that tensors and cotensors are adjoint to one another, and to establish other important adjunctions relating a tensored and cotensored $\V$-category $\D$ to the enriching category $\V$.  Before presenting these adjunctions in Proposition \ref{tensor-cotensor adjunctions} below, we define the notion of $\V$-adjunction that we use there. 
\begin{definition}
A $\V$-\emph{adjunction} (or $\V$-\emph{enriched adjunction}) between $\V$-categories $\A$ and $\B$ is a pair $F \colon \A \rightarrow \B$ and $G \colon \B \rightarrow \A$ of $\V$-functors together with a $\V$-natural isomorphism 
\[ \Map_{\B} \left( F,- \right) \cong \Map_{\A} \left( -, G \right). \]
\end{definition}

\begin{prop} \label{tensor-cotensor adjunctions} 
Let $\D$ be a $\V$-category that is both tensored and cotensored over $\V$.
\begin{enumerate}
\item For any object $D$ of $\D$ there is a $\V$-adjunction 
\[ \begin{tikzpicture}[node distance=3cm, auto]
\node (A) {$\V$};
\node (B) [right of=A] {$\D$.};
\draw[->, bend left=35] (A) to node {$D \otimes -$} (B);
\draw[->, bend left=35] (B) to node {$\Map_{\D}(D,-)$} (A);
\node (W) [node distance=1.5cm, right of=A] {$\perp$};
\end{tikzpicture} \]

\item For any object $E$ of $\D$ there is a $\V$-adjunction 
\[ \begin{tikzpicture}[node distance=3cm, auto]
\node (A) {$\D$};
\node (B) [right of=A] {$\V^{\op}$.};
\draw[->, bend left=35] (A) to node {$\Map_{\D}(-,E)$} (B);
\draw[->, bend left=35] (B) to node {$E^{(-)}$} (A);
\node (W) [node distance=1.5cm, right of=A] {$\perp$};
\end{tikzpicture} \]

\item For any object $V$ of $\V$ there is a $\V$-adjunction 
\[ \begin{tikzpicture}[node distance=3cm, auto]
\node (A) {$\D$};
\node (B) [right of=A] {$\D$.};
\draw[->, bend left=35] (A) to node {$- \otimes V$} (B);
\draw[->, bend left=35] (B) to node {$(-)^V$} (A);
\node (W) [node distance=1.5cm, right of=A] {$\perp$};
\end{tikzpicture} \]
\end{enumerate}
\end{prop}

\begin{proof} 
The first two adjunctions follow  directly from the definitions of tensor and cotensor (Definitions \ref{tensor def} and \ref{cotensor def}, respectively). The third adjunction follows from  both definitions, in particular for each pair of objects $D$ and $E$ in $\D$ we have 
\[ \Map_{\D}(D \otimes V,E) \cong \Map_{\D}(D, E)^V \cong \Map_{\D}(D, E^V). \] 
\end{proof}

We conclude with a standard example of a $\V$-category that is both tensored and cotensored over $\V$.

\begin{example} 
The category $\Top$ is enriched, tensored and cotensored in the category of simplicial sets $\s$. For topological spaces $X$ and $Y$, we define $\Map_{\Top}(X,Y)$ to be the simplicial set for which
\[ \Map_{\Top}(X,Y)_n = \Hom_{\Top}(X \times \Delta^n, Y), \]
where $\Delta^n$ denotes the standard topological $n$-simplex.  For a simplicial set $K$, let $|K|$ denote its geometric realization. There is an $\s$-functor
\[ \begin{tikzpicture}[node distance=4cm, auto]
\node (A) {$\Top \times \s$};
\node (B) [right of=A] {$\Top$};
\node (X) [node distance=1cm, below of=A] {$(X,K)$};
\node (Y) [right of=X] {$X \times |K|$};
\node (X') [node distance=2cm, below of=X] {$(Y,L)$};
\node (Y') [right of=X'] {$Y \times |L|$};
\node (x) [node distance=1cm, below of=X] {};
\node (x') [node distance=0.5cm, right of=x] {};
\node (y) [node distance=1cm, below of=Y] {};
\node (y') [node distance=0.5cm, left of=y] {};
\draw[->] (A) to node {} (B);
\draw[->] (X) to node [swap] {$\Map_{Top}(X,Y) \times L^K $} (X');
\draw[->] (Y) to node {$\Map_{\Top}(X \times |K|, \ Y \times |L|)$} (Y');
\draw[|->, dashed] (x') to node {} (y');
\end{tikzpicture} \]
which is defined in degree $n$ by the function
\[ \Hom_{\Top} \left(X \times \Delta^n, Y\right) \times \Hom_{\s}\left(K \times \Delta[n], L\right) \longrightarrow \Hom_{\Top}\left(X \times |K| \times \Delta^n, Y \times |L|\right) \]
which, using the fact that geometric realizations preserve finite products, is induced by the canonical function
\[ \Hom_{\Top}\left(X \times \Delta^n, Y\right) \times \Hom_{\Top}\left(|K| \times \Delta^n, |L|\right) \longrightarrow \Hom_{\Top}\left(X \times |K| \times \Delta^n, Y \times |L|\right). \]
Furthermore, for any topological spaces $X$ and $Y$ and simplicial set $K$ there is an isomorphism 
\[ \Map_{\Top}(X \times |K|, Y) \cong \Map_{\Top}(X,Y)^K, \]
so $\Top$ is tensored over $\s$. It can be shown analogously that $\Top$ is also cotensored over $\s$.
\end{example}

\section{The enriched Yoneda Lemma} \label{sectionyoneda}

In this section, we consider representable functors in the setting of enriched categories, and establish an enriched version of the Yoneda Lemma and its dual. We begin with enriched representable functors.  

\begin{definition}\label{representable}
Let $\C$ be a $\V$-category. For each object $C$ of $\C$, the $\V$-functor \emph{represented by} $C$ is given by
\[ \begin{tikzpicture}[node distance=4cm, auto]

\node (A) {$\C$};
\node (B) [right of=A] {$\V$};

\node (X) [node distance=1cm, below of=A] {$A$};
\node (Y) [right of=X] {$\Map_{\C}(C,A)$};
\node (X') [node distance=2cm, below of=X] {$B$};
\node (Y') [right of=X'] {$\Map_{\C}(C,B)$};
\node (x) [node distance=1cm, below of=X] {};
\node (x') [node distance=0.5cm, right of=x] {};
\node (y) [node distance=1cm, below of=Y] {};
\node (y') [node distance=0.5cm, left of=y] {};

\draw[->] (A) to node {$R^C$} (B);
\draw[->] (X) to node [swap] {$\Map_{\C} (A,B)$} (X');
\draw[->] (Y) to node {$\Map_{\C} (C,B) ^{\Map_{\C} (C,A)}$} (Y');
\draw[|->, dashed] (x') to node {} (y');

\end{tikzpicture} \]
where the morphism 
\[ \Map_\C(A,B) \rightarrow \Map_\C(C,B)^{\Map_\C(C,A)}, \]
which we sometimes call the \emph{Yoneda embedding}, is the one adjoint to the composition morphism in $\C$
\[ \Map_\C(C,A) \otimes \Map_\C(A,B) \rightarrow \Map_\C(C,B) \]
via part (3) of Proposition \ref{tensor-cotensor adjunctions}.
\end{definition}

Before stating the Yoneda Lemma, we need the definition of an end, for which we  make some observations.  Suppose that $\C$ is a $\V$-category and that $X \colon \C^{\op} \times \C \rightarrow \V$ is a $\V$-bifunctor.  Given two objects $A$ and $B$ of $\C$, there is an induced morphism
\[ X(A,-) \colon \Map_\C(A,B) \rightarrow X(A,B)^{X(A,A)} \]
in $\V$.  Using the adjunction between the tensor and cotensor in $\V$, such a morphism corresponds to a morphism
\[ X(A,A) \otimes \Map_\C(A,B) \rightarrow X(A,B), \]
which in turn corresponds to a morphism
\[ X(A,A) \rightarrow X(A,B)^{\Map_\C(A,B)}. \]
One can analogously produce a morphism 
\[ X(B,B) \rightarrow X(A,B)^{\Map_\C(A,B)} \]
from the natural map
\[ X(-,B) \colon \Map_\C(A,B) \rightarrow X(A,B)^{X(B,B)}. \] 
\begin{definition}  \label{end}
Assume that $\V$ is complete, and let $X \colon \C^{\op} \times \C \rightarrow \V$ be a $\V$-bifunctor.   The \emph{end} of $X$ is the equalizer
\[ \int_A X(A,A) \rightarrow \prod_A X(A,A) \rightrightarrows \prod_{A,B} X(A,B)^{\Map_\C(A,B)} \]
in $\V$ whose parallel morphisms are induced by the ones described above. 
\end{definition}

We want to make a similar construction from a pair of $\V$-functors $F \colon \C \rightarrow \D$ and $X \colon \C \rightarrow \V$.  Assuming that $\D$ is complete and tensored and contensored over $\V$, we want to construct parallel morphisms
\[\prod_{A} FA^{XA} \rightrightarrows \prod_{A,B} {(FB^{XA})}^{\Map_{\C}(A,B)} \]
in $\D$.  We can take the first morphism on each factor to be given by the map 
\[ FA^{XA} \rightarrow {(FB^{XA})}^{\Map_{\C}(A,B)} \]
which corresponds under adjunction to the composite 
\[ \Map_{\C}(A,B) \otimes XA \otimes FA^{XA} \rightarrow \Map_{\D}(FA,FB) \otimes FA \rightarrow FB \]
where the first arrow uses  the counit of adjunction (3) of Proposition \ref{tensor-cotensor adjunctions} and the $\V$-functor structure of $F$, and the second arrow uses the counit of adjunction (1) of Proposition \ref{tensor-cotensor adjunctions}. We can obtain the second parallel morphism
\[ FB^{XB} \rightarrow {(FB^{XA})}^{\Map_{\C}(A,B)} \]
similarly, namely as the one corresponding via adjunction to the composite
\[ \Map_{\C}(A,B) \otimes XA \otimes FB^{XB} \rightarrow XB^{XA} \otimes XA \otimes FB^{XB} \rightarrow XB \otimes FB^{XB} \rightarrow FB. \]

\begin{definition} \label{Fend}
Assume that $\D$ is complete and tensored and cotensored over $\V$, and let $F \colon \C \rightarrow \D$ and $X \colon \C \rightarrow \V $ be $\V$-functors.  The \emph{evaluated cotensor} of the pair $(F,X)$ is the equalizer  
\[ F^X := \int_{A} FA^{XA} \rightarrow \prod_{A} FA^{XA} \rightrightarrows \prod_{A,B} {(FB^{XA})}^{\Map_{\C}(A,B)} \]
in $\D$ whose parallel morphisms are the ones described above.  
\end{definition}

\begin{remarks} \label{cotensor} \label{brremark}
Let us make a few comments about the notation and terminology we have chosen here.
\begin{enumroman}
\item Observe that, as suggested by the notation 
\[F^X = \int_{A} FA^{XA}, \]
the evaluated cotensor can be thought of as a generalization of an ordinary end.  In the special case when $\D = \V$, the assignment $F^{X}(A,B) = FB^{XA}$ defines a $\V$-bifunctor $\C^{\op} \times \C \rightarrow \V$ whose end is precisely $F^X$. 

\item However, we have chosen to use the word ``cotensor" to describe this construction.  We justify this terminology in the next section by proving that it behaves much like a cotensor.  In particular, in Propositions \ref{FunAdjunctionX} and \ref{FunAdjunctionF}, we show that the assignment 
\[ (F,X)\mapsto F^X \]
defines a $\V$-bifunctor $\Fun(\C,\V)^{\op} \times \Fun(\C,\D) \rightarrow \D$.  

\item In \cite[2.5]{br}, Biedermann and R\"ondigs use the notation $\mathbf{hom}(X,F)$ for the evaluated cotensor $F^X$.  
\end{enumroman}
\end{remarks}
  
Setting $X=R^C$ in the evaluated cotensor yields the following enriched version of the Yoneda lemma. 

\begin{lemma} \label{Yoneda} 
Let $\C$ and $\D$ be $\V$-categories with $\D$ complete and both tensored and cotensored over $\V$, and let $F \colon \C \rightarrow \D$ be a $\V$-functor.  For each object $C$ of $\C$ there is a natural isomorphism 
\[ FC \cong F^{R^C} = \int_{A} FA^{R^C(A)} = \int_{A} FA^{\Map_{\C}(C,A)}. \]
\end{lemma}

\begin{proof}
The proof follows an argument similar to the one outlined for the case where $\D=\V$ in \cite[\S 2.4]{kelly}.  The key step is noting that for each object $A$ in $\C$ there is a morphism
\[ FC \rightarrow FA^{\Map_{\C}(C,A)} \]
corresponding under adjunction to 
\[ F_{CA}:\Map_{\C}(C,A) \rightarrow \Map_{\D}(FC,FA) \]
and then proving that the collection of all such morphisms exhibits $FC$ as the equalizer of the parallel pair
\[ \prod_{A} FA^{\Map_{\C}(C,A)} \rightrightarrows \prod_{A,B} {(FB^{\Map_{\C}(C,A)})}^{\Map_{\C}(A,B)}. \] 
The result then follows by an application of Definition \ref{Fend}.
\end{proof}

For future use, we include the following result showing that cotensors commute with evaluated contensors.

\begin{proposition} \label{l:homcotensor} 
Supppose that $\D$ is complete and tensored and cotensored over $\V$.  For any $\V$-functors $F \colon \C \rightarrow \D$ and $X \colon \C \rightarrow \V$ and any object $V$ of $\V$, there are natural isomorphisms in $\V$
\[ \int_\C FC^{XC \otimes V} \cong \int_\C \left(FC^{XC}\right)^V \cong \left (\int_\C FC^{XC} \right)^V. \]
\end{proposition}

\begin{proof}
The first isomorphism follows from the adjunction in part (3) of Proposition \ref{tensor-cotensor adjunctions}.  That same proposition also shows that the cotensor defines a right adjoint, and since ends are limits and right adjoints preserve limits, the second isomorphism follows. 
\end{proof}

To finish this section, we give some dual constructions and results, leading to the co-Yoneda Lemma.  

First,  we can consider the \emph{co-Yoneda embedding}, which is the assignment
\[ \Map_\C(B,A) \rightarrow \Map_\C (B,C)^{\Map_\C(A,C)}, \]
that is adjoint to the composition morphism
\[ \Map_\C(B,A) \otimes \Map_\C(A,C) \rightarrow \Map_\C(B,C). \]
It gives rise to the $\V$-functor
\[ \begin{tikzpicture}[node distance=4cm, auto]

\node (A) {$\C^{\op}$};
\node (B) [right of=A] {$\V$};

\node (X) [node distance=1cm, below of=A] {$A$};
\node (Y) [right of=X] {$\Map_{\C}(A,C)$};
\node (X') [node distance=2cm, below of=X] {$B$};
\node (Y') [right of=X'] {$\Map_{\C}(B,C)$.};
\node (x) [node distance=1cm, below of=X] {};
\node (x') [node distance=0.5cm, right of=x] {};
\node (y) [node distance=1cm, below of=Y] {};
\node (y') [node distance=0.5cm, left of=y] {};

\draw[->] (A) to node {$R_C$} (B);
\draw[->] (X) to node [swap] {$\Map_{\C} (B,A)$} (X');
\draw[->] (Y) to node {$\Map_{\C} (B,C) ^{\Map_{\C} (A,C)}$} (Y');
\draw[|->, dashed] (x') to node {} (y');

\end{tikzpicture} \]

We can also define the notion of a coend of a $\V$-bifunctor $\C^{\op} \times \C \rightarrow \V$.

\begin{definition}
Assume that $\V$ is cocomplete. The \emph{coend} of a $\V$-bifunctor $X \colon \C^{\op} \times \C \rightarrow \V$ is the coequalizer
\[ \coprod_{A,B} X(A,B) \otimes \Map_\C(B,A) \rightrightarrows \coprod_A X(A,A) \rightarrow \int^A X(A,A) \]
in $\V$ whose morphisms can be obtained similarly to the ones used in the definition of end. 
\end{definition}

We can make a similar definition for a pair of $\V$-functors $F \colon \C \rightarrow \D$ and $X \colon \C \rightarrow \V^{\op}$, but we do not distinguish it with another name and simply refer to it as a coend.  In contrast to the discussion in Remark \ref{cotensor}, while this construction is formally dual to that of the evaluated cotensor, it does not seem to have properties that would justify calling it an ``evaluated tensor". In particular, the assignment $\Fun(\C, \D) \times \Fun(\C, \V^{\op}) \rightarrow \D$ constructed in Definition \ref{Fcoend} does not appear to be $\V$-functorial.

\begin{definition}  \label{Fcoend}
Assume that $\D$ is cocomplete and tensored and cotensored over $\V$, and let $F \colon \C \rightarrow \D$  and $X \colon \C \rightarrow \V^{\op}$ be $\V$-functors. The \emph{coend}  of $(F, X)$ is the coequalizer in $\D$ of the parallel morphisms
\[\coprod_{A,B} FB \otimes XA \otimes \Map_{\C}(B,A) \rightrightarrows \coprod_A FA \otimes XA \rightarrow \int^{A} FA \otimes XA \]
analogous to those in Definition \ref{Fend}.
\end{definition}

The following result, which we call the \emph{co-Yoneda Lemma}, can be proved analogously to the Yoneda Lemma.

\begin{lemma} \label{Coyoneda}
Let $\C$ and $\D$ be $\V$-categories with $\D$ cocomplete and both tensored and cotensored in $\V$, and let $F \colon \C \rightarrow \D$ be a $\V$-functor. For each object $C$ in $\C$, there is a natural isomorphism
\[ FC \cong \int^{A} FA \otimes R_C(A) = \int^{A} FA \otimes \Map_{\C}(A,C). \]
\end{lemma}

\section{Enriched functor categories} \label{sectionfunctors}

Because our motivation arises from functor calculus, we are primarily interested in categories whose objects are themselves given by functors between categories.  Thus, let $\V$ be a complete closed monoidal category, 
and let $\C$ and $\D$ be $\V$-categories with $\C$ small. 
We consider the category $\Fun(\C,\D)$ whose objects are $\V$-functors $\C \rightarrow \D$ and whose morphisms are $\V$-natural transformations.  

We emphasize that, although the notation $\Fun(\C,\D)$ is often used to denote the category of all functors $\C \rightarrow \D$ with no additional structure, in this paper we always use it to denote the category of $\V$-functors and $\V$-natural transformations.

\begin{prop}\cite[Section 5]{day2} 
The category $\Fun(\C,\D)$ is a $\V$-category.  For $\V$-functors $F, G \colon \C \rightarrow \D$, the mapping object is given by
\[ \Map_{\Fun(\C,\D)}(F,G): = \int_A \Map_{\D}(FA,GA) \]
where $A$ ranges over all objects of $\C$.  
\end{prop}

\begin{proof}
To see the mapping object indeed defines the structure of a $\V$-category, first observe that for any $\V$-functors $F,G,H \colon \C \rightarrow \D$, we can define composition morphisms
\[ \int_A \Map_\D(FA,GA) \otimes \int_A \Map_D(GA,HA) \rightarrow \int_A \Map_\D(FA,HA) \]
using the composition morphisms
\[ \Map_\D(FA, GA) \otimes \Map_\D(GA,HA) \rightarrow \Map_\D(FA, HA) \]
in $\D$ for any object $A$ of $\C$. The identity morphism at a $\V$-functor $F$ is the morphism 
\[ I \rightarrow \int_{A} \Map_{\D}(FA,FA) \]
induced by the collection of identity morphisms 
\[i \colon I \rightarrow \Map_{\D}(FA,FA) \]
in $\D$ for each object $A$ of $\C$.

To complete the proof, we need to show that the ordinary category underlying this $\V$-category is the category of $\V$-functors $\Fun(\C,\D)$; see Definition \ref{underlyingcat}.  Observe that a morphism
\[ I \rightarrow \int_{A} \Map_{\D}(FA,GA) \]
in $\V$ is equivalently a morphism  
\[ I \rightarrow \prod_{A} \Map_{\D}(FA,GA) \]
such that the morphisms obtained from composing with the parallel pair
\[ \prod_A \Map_\D(FA, GA) \rightrightarrows \prod_{A,B} \Map_\D(FA, GB)^{\Map_\C(A,B)} \]  
are equal. By unpacking the definitions of this parallel pair, we see that such a map is given by a collection of morphisms 
\[ \alpha_A \colon I \rightarrow \Map_{\D}(FA,GA), \]
where $A$ ranges over all objects of $\C$, such that the following diagram commutes for all objects $A$ and $B$ of $\C$: 
\[ \begin{tikzpicture}[node distance=5cm, auto]

\node (A) {\small $I \otimes \Map_{\C}(A,B)$};
\node (B) [node distance=5cm, right of=A] {\small $\Map_{\D}(FA,GA) \otimes \Map_{\C}(A,B)$};
\node (C) [node distance=6.5cm, right of=B] {\small $\Map_{\D}(FA,GA) \otimes \Map_{\D}(GA,GB)$}; 
\node (A') [node distance=1cm, below of=A] {\small $\Map_{\C}(A,B) \otimes I$};
\node (X) [node distance=2.5cm, below of=A] {\small $\Map_{\C}(A,B) \otimes \Map_{\D}(FB,GB)$};
\node (Y) [node distance=6.5cm, right of=X] {\small $\Map_{\D}(FA,GB) \otimes \Map_{\D}(FB,GB)$};
\node (Z) [right of=Y] {\small $\Map_{\D}(FA,GB)$.};

\draw[->] (A) to node {\small $\alpha_A \otimes 1$} (B);
\draw[->] (B) to node {\small $1 \otimes G_{AB}$} (C);
\draw[->] (X) to node [swap] {\small $F_{AB} \otimes 1$} (Y);
\draw[->] (Y) to node [swap] {\small $\circ$} (Z);
\draw[->] (A) to node [swap] {} (A');
\draw[->] (A') to node [swap] {\small $1 \otimes \alpha_B$} (X);
\draw[->] (C) to node {\small $\circ$} (Z);

\end{tikzpicture} \]
Comparing with Definition \ref{V-natural transformations}, we that there is a natural isomorphism of sets 
\[ \Hom_{\V}\left(I, \int_{A} \Map_{\D}(FA,GA) \right) \cong \Hom_{\Fun(\C, \D)}(F,G). \] 
Thus $\Fun(\C, \D)$ is a $\V$-category with the indicated mapping objects.
\end{proof}

We now prove that, if $\D$ is either tensored or cotensored over $\V$, then that same structure is transferred to $\Fun(\C,\D)$.

\begin{prop} \label{Fun is tensored} 
If $\D$ is tensored over $\V$, then $\Fun(\C, \D)$ is also tensored over $\V$, with tensor structure defined by $(F \otimes V)(A) = FA \otimes V$ for each object $A$ in $\C$.
\end{prop}

\begin{proof} 
More specifically, we define the $\V$-functor
\[ \begin{tikzpicture}[node distance=4cm, auto]

\node (A) {$\Fun(\C, \D) \times \V$};
\node (B) [right of=A] {$\Fun(\C, \D)$};

\node (X) [node distance=1cm, below of=A] {$(F,V)$};
\node (Y) [right of=X] {$F \otimes V$};
\node (X') [node distance=2cm, below of=X] {$(G,W)$};
\node (Y') [right of=X'] {$G \otimes W$};
\node (x) [node distance=1cm, below of=X] {};
\node (x') [node distance=0.5cm, right of=x] {};
\node (y) [node distance=1cm, below of=Y] {};
\node (y') [node distance=0.5cm, left of=y] {};

\draw[->] (A) to node {$- \otimes -$} (B);
\draw[->] (X) to node [swap] {$\Map_{\Fun(\C, \D)}(F, G) \otimes W^V$} (X');
\draw[->] (Y) to node {$\Map_{\Fun(\C, \D)} \left(F \otimes V, G \otimes W \right)$} (Y');
\draw[|->, dashed] (x') to node {} (y');
\end{tikzpicture} \]
where $(F \otimes V)(A) = FA \otimes V$ for each object $A$ in $\C$.

To show that $F \otimes V$ is a $\V$-functor $\C \rightarrow \D$ we have to define, for any objects $A$ and $B$ of $\C$, compatible morphisms
\[ \Map_\C(A,B) \rightarrow \Map_\D(FA \otimes V, FB \otimes V) \]
in $\V$. Such morphisms are given by composites
\[ \xymatrixcolsep{4pc}\xymatrix{\Map_\C(A,B) \cong \Map_\C(A,B) \otimes I \ar[r]^-{F_{AB} \otimes i} \ar[dr] & \Map_\D(FA, FB) \otimes V^V \ar[d] \\ 
& \Map_\D(FA \otimes V, FB \otimes V),} \]
where the map $i$ was described in Remark \ref{mapi} and the downward morphism comes from applying Definition \ref{tensor def} in $\D$.

Given an object $A$ of $\C$, we can apply Definition \ref{tensor def} to $D = FA$ and $E=GA$ to obtain a morphism in $\V$
\[ \Map_\D(FA, GA) \otimes W^V \rightarrow \Map_\D(FA \otimes V, GA \otimes W). \]
As $A$ varies, we obtain an assignment
\[ \Map_{\Fun(\C, \D)}(F, G) \otimes W^V \rightarrow \Map_{\Fun(\C, \D)} \left(F \otimes V, G \otimes W \right) \]
so $- \otimes -:\Fun(\C,\D) \otimes \V \rightarrow \Fun(\C,\D)$ is itself a $\V$-functor.

Finally, the isomorphisms 
\[ \Map_{\D}(FA \otimes V, GA) \cong \Map_{\D}(FA, GA)^V \]
in $\D$ induce the necessary corresponding isomorphisms
\[ \Map_{\Fun(\C, \D)}(F \otimes V, G) \cong \Map_{\Fun(\C, \D)}(F, G)^V \]
in $\Fun(\C,\D)$.
\end{proof}

We now prove the analogous result when $\D$ is cotensored in $\V$.

\begin{prop} \label{Fun is cotensored} 
If $\D$ is cotensored over $\V$, then so is $\Fun(\C,\D)$, via $(F^V)(A) = FA^V$ for each object $A$ in $\C$.
\end{prop}

\begin{proof} 
We claim that the definition of the cotensor we have given for objects defines a bifunctor
\[ \Fun(\C,\D) \times \V^{\op} \rightarrow \Fun(\C,\D). \]
Setting $D = FA$ and $E=GA$ for each object $A$ of $\C$ in Definition \ref{cotensor def}, we get a morphism
\[ \Map_\D(FA, GA) \otimes V^W \rightarrow \Map_\D(FA^V, GA^W) \]
in $\V$. As we take ends over all objects $A$
we obtain the following assignment.
\[ \begin{tikzpicture}[node distance=4cm, auto]
\node (A) {$\Fun(\C, \D) \times \V^{\op}$};
\node (B) [right of=A] {$\Fun(\C, \D)$};

\node (X) [node distance=1cm, below of=A] {$(F,V)$};
\node (Y) [right of=X] {$F^V$};
\node (X') [node distance=2cm, below of=X] {$(G,W)$};
\node (Y') [right of=X'] {$G^W$};
\node (x) [node distance=1cm, below of=X] {};
\node (x') [node distance=0.5cm, right of=x] {};
\node (y) [node distance=1cm, below of=Y] {};
\node (y') [node distance=0.5cm, left of=y] {};

\draw[->] (A) to node {$(-)^{(-)}$} (B);
\draw[->] (X) to node [swap] {$\Map_{\Fun(\C, \D)}(F, G) \otimes V^W$} (X');
\draw[->] (Y) to node {$\Map_{\Fun(\C, \D)} \left(F^V, G^W \right)$} (Y');
\draw[|->, dashed] (x') to node {} (y');

\end{tikzpicture} \]
To see that $F^V$ is a $\V$-functor $\C \rightarrow \D$ note that for objects $A$ and $B$ of $\C$, the map
\[ \Map_\C(A,B) \rightarrow \Map_\D\left(FA^V, FB^V\right) \]
is given by the composite 
\[ \xymatrix{\Map_\C(A,B) \cong \Map_\C(A,B) \otimes I \ar[rr]^-{F_{AB} \otimes i} \ar[drr] && \Map_\D(FA, FB) \otimes V^V \ar[d] \\
&& \Map_\D(FA^V, FB^V),} \]
where the downward morphism is another instance of the one in Definition \ref{cotensor def}.  

Finally, we obtain an isomorphism
\[ \Map_{\Fun(\C, \D)}(F, G^V) \cong \Map_{\Fun(\C, \D)}(F, G)^V \]
from the isomorphisms
\[ \Map_{\D}(FA, GA^V) \cong \Map_{\D}(FA, GA)^V \]
in $\V$ for every object $A$ of $\C$.
\end{proof}

One might also wonder if $\Fun(\C,\D)$ is enriched, tensored, and cotensored in $\Fun(\C,\V)$, and in particular, if the evaluated cotensor of Definition \ref{Fend} serves as a cotensor in this context. Having such an enrichment would require $\Fun(\C,\V)$ to be a closed monoidal category.  While Day \cite{day} shows that $\Fun(\C,\V)$ has this additional structure for certain $\V$-categories $\C$ satisfying some further restrictive conditions, it is not true in general.  

For our work, we do not need this further enrichment, and we prefer to avoid imposing extra conditions on $\C$ to obtain it.  We do, however, want to make use of adjunctions similar to those of parts (2) and (3) of Proposition \ref{tensor-cotensor adjunctions} for the evaluated cotensor. We thus establish the existence of such adjunctions via the next two propositions.  While we do not need it for our purposes, for the curious reader we also include a third result establishing the analogue of part (1) of Proposition \ref{tensor-cotensor adjunctions}.

The proofs of these three results are very similar in structure, so we only include a full argument for the first of them.  For the remaining two we simply state the result and leave the details to the reader. 

We begin with the adjunction analogous to part (3) of Proposition \ref{tensor-cotensor adjunctions}, in which the evaluated cotensor is used to define the right adjoint.  

\begin{proposition}\label{FunAdjunctionX}
For each $\V$-functor $X \colon \C \rightarrow \V$ there is a $\V$-adjunction
\[ \begin{tikzpicture}[node distance=3.6cm, auto]
\node (A) {$\D$};
\node (B) [right of=A] {$\Fun(\C, \D)$};
\draw[->, bend left=35] (A) to node {$X \otimes -$} (B);
\draw[->, bend left=35] (B) to node {$(-)^X$} (A);
\node (W) [node distance=1.8cm, right of=A] {$\perp$};
\end{tikzpicture} \] 
where $(X \otimes D)(C) := XC \otimes D$ for each object $C$ of $\C$. 
\end{proposition}

\begin{proof}
First we show that $X \otimes -$ is well defined on objects, i.e., that $X \otimes D$ is a $\V$-functor $\C \rightarrow \D$ for each object $D$ of $\D$. In order to do so, we need to construct a morphism 
\[ \Map_{\C}(A,B) \rightarrow \Map_{\D}(XA \otimes D, XB \otimes D) \]
in $\V$ for each pair of objects $A$ and $B$ in $\C$. Such a morphism is given by the composite
\[ \Map_{\C}(A,B) \rightarrow XB^{XA} \rightarrow  \Map_{\D}(D, XB \otimes D)^{XA} \cong \Map_{\D}(XA \otimes D, XB \otimes D) \]
where the second morphism is induced by the unit of the adjunction in part (1) of Proposition \ref{tensor-cotensor adjunctions}, applied to the component at $XB$. 

To see that $X \otimes -$ defines a $\V$-functor $\D \rightarrow \Fun(\C,\D)$, note that for any pair of objects $D$ and $E$ in $\D$ and object $C$ in $\C$ we have a morphism 
\[ \Map_{\D}(D,E) \rightarrow \Map_{\D} \left(D, (XC \otimes E)^{XC} \right) \cong \Map_{\D}(XC \otimes D, XC \otimes E), \] 
in $\V$ given by the unit of the adjunction in part (3) of Proposition \ref{tensor-cotensor adjunctions}. The collection of all such morphisms induces a map 
\[ \Map_{\D}(D,E) \rightarrow \Map_{\Fun(\C, \D)}(X \otimes D, X \otimes E). \]

Next, we show that $(-)^X$ defines a $\V$-functor $\Fun(\C,\D) \rightarrow \D$, for which we must construct a morphism
\[ \Map_{\Fun(\C, \D)}(F,G) \rightarrow \Map_{\D}(F^X, G^X) = \Map_{\D}\left(\int_{\C}FC^{XC}, \int_{\C}GC^{XC}\right) \]
in $\V$ for each pair of $\V$-functors $F, G \colon \C \rightarrow \D$. Observe that for each object $C$ in $\C$ there is a composite
\[ \Map_{\D}(FC,GC) \rightarrow \Map_{\D}(XC \otimes FC^{XC}, GC) \cong \Map_{\D}(FC^{XC}, GC^{XC}) \rightarrow \Map_{\D}\left(\int_C FC^{XC}, GC^{XC}\right) \] 
where the first map given by the counit of the adjunction in part (3) of Proposition \ref{tensor-cotensor adjunctions} and the second map is induced by the projection $\displaystyle\int_C FC^{XC} \rightarrow FC^{XC}$. The collection of these morphisms induces
\[ \Map_{\Fun(\C, \D)}(F,G) = \int_C \Map_{\D}(FC,GC) \rightarrow \int_C \Map_{\D}\left(\int_C FC^{XC}, GC^{XC}\right), \]
and since $\Map_{\D}(D,-)$ is a right adjoint for all $D$ and therefore preserves limits, we have
\[ \int_C \Map_{\D}\left(\int_C FC^{XC}, GC^{XC}\right) \cong \Map_{\D}\left(\int_C FC^{XC}, \int_C GC^{XC} \right). \]

Finally, we show that the $\V$-functors in the Proposition form an adjoint pair by constructing a unit and a counit satisfying the triangle identities. The unit of the adjunction at an object $D$ of $\D$ is the morphism
\[ D \rightarrow \int_{\C}(XA \otimes D)^{XA} = (X \otimes D)^X\]
induced by the collection of morphisms
\[ D \rightarrow (XA \otimes D)^{XA} \]
given by unit of the adjunction of part (3) of Proposition \ref{tensor-cotensor adjunctions}.

The counit of the adjunction at a $\V$-functor $F \colon \C \rightarrow \D$ is the natural transformation 
\[ X \otimes \int_{\C}FC^{XC} \rightarrow F \]
induced by the composites
\[ XA \otimes \int_{\C}FC^{XC} \rightarrow XA \otimes \prod_{\C} FC^{XC} \rightarrow XA \otimes FA^{XA}  \rightarrow FA \]
for each object $A$ of $\C$, where the rightmost arrow is the counit of the adjunction in part (3) of Proposition \ref{tensor-cotensor adjunctions}.  It is now straightforward to show that the triangle identities are satisfied. 
\end{proof}

The next result provides a similar analogue to part (2) of Proposition \ref{tensor-cotensor adjunctions}. 

\begin{proposition} \label{FunAdjunctionF}
For each $\V$-functor $F \colon \C \rightarrow \D$ there is a $\V$-adjunction
\[ \begin{tikzpicture}[node distance=3.6cm, auto]
\node (A) {$\D$};
\node (B) [right of=A] {$\Fun(\C, \V)^{\op}$,};
\draw[->, bend left=35] (A) to node {$\Map_\D(-,F)$} (B);
\draw[->, bend left=35] (B) to node {$F^{(-)}$} (A);
\node (W) [node distance=1.8cm, right of=A] {$\perp$};
\end{tikzpicture} \] 
where $\Map_{\D}(D,F)(A) = \Map_{\D}(D,FA)$.
\end{proposition}

To end this section, we provide an adjunction in which tensoring with an object in $\V$ in part (1) of Proposition \ref{tensor-cotensor adjunctions} is replaced by tensoring with an object in $\Fun(\C,\V).$ 

Recall the definitions of $(X \otimes D)$ and $\Map_{\D}(D,F)$ from Propositions \ref{FunAdjunctionX} and \ref{FunAdjunctionF}, respectively.  

\begin{proposition}\label{FunAdjunctionD}
For each object $D$ of $\D$ there is a $\V$-adjunction
\[ \begin{tikzpicture}[node distance=3.6cm, auto]
\node (A) {$\Fun(\C, \V)$};
\node (B) [right of=A] {$\Fun(\C, \D)$.};
\draw[->, bend left=35] (A) to node {$- \otimes D$} (B);
\draw[->, bend left=35] (B) to node {$\Map_\D(D,-)$} (A);
\node (W) [node distance=1.8cm, right of=A] {$\perp$};
\end{tikzpicture} \] 
\end{proposition}

\section{Enriched model categories} \label{sectionmodel}

We now turn to $\V$-categories that are equipped with the additional structure of a model category.  Recall that a model category is a category $\C$ with three distinguished classes of morphisms called \emph{cofibrations}, \emph{fibrations}, and \emph{weak equivalences}, satisfying five axioms \cite[7.1.3]{hirschhorn}. An \emph{acyclic fibration} is a fibration which is also a weak equivalence, and similarly an \emph{acyclic cofibration} is both a cofibration and a weak equivalence. 

For example, the category $\s$ of simplicial sets has a model structure with weak equivalences given by the maps whose geometric realizations are weak homotopy equivalences, fibrations the Kan fibrations, and cofibrations the monomorphisms \cite{quillen}. 

An object $A$ in a model category $\C$ is \emph{cofibrant} if the unique morphism from the initial object in $\C$ to $A$ is a cofibration. Dually, an object $X$ is \emph{fibrant} if the unique morphism from $X$ to the final object in $\C$ is a fibration.  Given a model structure on a category $\C$ we often have a good understanding of the fibrant or the cofibrant objects, and sometimes both. The model structure on a category $\C$ is usually set up so that the cofibrant or the fibrant objects are the primary objects of interest.  For instance, in the standard model structure on the category of simplicial sets $\s$ all objects are cofibrant and Kan complexes are the fibrant objects. 

In nice examples, a model category has the additional structure of being \emph{cofibrantly generated} \cite[\S 11.1]{hirschhorn}, in that there are sets $I$ and $J$ of maps such that a map is an acyclic fibration if and only if it has the right lifting property with respect to the maps in $I$, and a fibration if and only if it has the right lifting property with respect to $J$, with both sets satisfying the small object argument \cite[10.5.16]{hirschhorn}.  In the example $\s$ the set $I$ can be taken to be the set of boundary inclusions, and the set $J$ can be taken to be the set of horn inclusions.  

As before, let $(\V,\otimes,I)$  be a closed symmetric monoidal category but now we require it to also have the structure of a model category.  In particular, we want these two structures on $\V$ to be compatible, in the following sense.

\begin{definition} 
A \emph{symmetric monoidal model category} $\V$ is a symmetric monoidal category $(\V, \otimes, I)$ together with a model category structure on $\V$ satisfying the following conditions. 
\begin{itemize}
    \item If $i \colon K \rightarrow L$ and $j \colon V \rightarrow W$ are cofibrations in $\V$, then 
    \[ K\otimes W \cup_{K \otimes V} L \otimes V \rightarrow L \otimes W \] 
    is a cofibration and a weak equivalence if  $i$ or $j$ is. 
    \item If $QI$ is a cofibrant replacement of $I$, then for any object $V$ of $\V$ the induced map  $QI \otimes V \rightarrow I \otimes V$ is a weak equivalence.
\end{itemize}
\end{definition} 

\begin{example}
We have seen in Example \ref{ssetissymmon} that the category $\s$ of simplicial sets is a closed symmetric monoidal category with the monoidal structure given by the usual product.  This structure is compatible with the model structure described above \cite[4.2.8]{hovey}.
\end{example}

Next, we consider what it means for the model structure on $\D$ and the $\V$-enrichment of $\D$ to be compatible with one another. We label the axioms below according to the usual convention when $\V$ is the category of simplicial sets, i.e., in the definition of a simplicial model category.

\begin{definition} 
Let $\V$ be a closed symmetric monoidal model category.  A $\V$-\emph{model category} is a $\V$-category $\D$ that is equipped with a model structure on its underlying category such that
\begin{itemize}
\item[(MC6)] the category $\D$ is tensored and cotensored over $\V$; and

\item[(MC7)] for any fibration $p \colon D \rightarrow E$ in $\D$ and cofibration $i \colon V \rightarrow W$ in $\V$, the pullback corner map 
\[ \begin{tikzpicture}[node distance=3cm, auto]

\node (A) {$D^V \times_{E^V} E^W$};
\node (B) [right of=A] {$D^V$};
\node (C) [node distance=2.2cm, below of=A] {$E^W$};
\node (D) [right of=C] {$E^V$};
\node (E') [node distance=1.2cm, above of=A, left of=A] {};
\node (E) [node distance=0.4cm, left of=E'] {$D^W$};

\draw[->] (A) to node {} (B);
\draw[->] (C) to node [swap] {} (D);
\draw[->] (A) to node [swap] {} (C);
\draw[->] (B) to node {} (D);
\draw[->, bend right] (E) to node {} (C);
\draw[->, bend left=18] (E) to node {} (B);
\draw[->, dashed] (E) to node [swap]  {} (A);

\end{tikzpicture} \]
is a fibration, and is a weak equivalence if either $i$ or $p$ is. 
\end{itemize} 
\end{definition}

\begin{remark} 
Assuming that $\D$ is a model category and (MC6) holds, (MC7) is equivalent to the condition
\begin{enumerate}[(MC7')]
\item for any cofibration $i \colon D \rightarrow E$ in $\D$ and cofibration $j \colon V \rightarrow W$ in $\V$, the pushout corner map 
\[ \begin{tikzpicture}[node distance=3.5cm, auto]

\node (A) {$D \otimes V$};
\node (B) [right of=A] {$D \otimes W$};
\node (C) [node distance=2.2cm, below of=A] {$E \otimes V$};
\node (D) [right of=C] {$(D \otimes W) \amalg_{D \otimes V} (E \otimes V)$};
\node (E') [node distance=1.4cm, below of=D, right of=D] {};
\node (E) [node distance=1cm, right of=E'] {$E \otimes W$};

\draw[->] (A) to node {} (B);
\draw[->] (C) to node [swap] {} (D);
\draw[->] (A) to node [swap] {} (C);
\draw[->] (B) to node {} (D);
\draw[->, bend right=18] (C) to node {} (E);
\draw[->, bend left=35] (B) to node {} (E);
\draw[->, dashed] (D) to node [swap]  {} (E);

\end{tikzpicture} \]
is a cofibration, and is a weak equivalence if either $i$ or $j$ is. 
\end{enumerate}

This statement is proved in \cite[9.3.7]{hirschhorn} when $\V$ is the category of simplicial sets. The proof is analogous for more general $\V$. 
\end{remark}

\begin{convention} 
From now on we assume that $\C$ is a small $\V$-category and that $\D$ is a $\V$-model category. 
\end{convention}
      
The following result tells us that given any cofibrantly generated $\V$-model category $\D$, there exists a model structure on $\Fun(\C,\D)$ induced by the model structure on $\D$. We omit some of the technical assumptions, since we do not need them here, but refer the reader to \cite[4.32]{gm} for the precise statement.  See also \cite{dro}, \cite{ds}, and \cite{schwedeshipley} for the special case when $\D=\V$, using slightly different model category assumptions.

\begin{theorem} \label{projmodel} 
Let $\C$ be a small $\V$-category and $\D$ be a cofibrantly generated $\V$-model category satisfying some mild conditions on the sets of generating cofibrations and acyclic cofibrations.  Then the category $\Fun(\C,\D)$ has a model structure in which a morphism $F \rightarrow G$ in $\Fun(\C,\D)$ is:
\begin{itemize} 
\item a weak equivalence if $FA \rightarrow GA$ is a weak equivalence in $\D$ for all objects $A$ of $\C$;

\item a fibration if $FA \rightarrow GA$ is a fibration in $\D$ for all objects $A$ of $\C$; and

\item a cofibration if it has the left lifting property with respect to all acyclic fibrations in $\Fun(\C,\D)$. 
\end{itemize}
This model structure on $\Fun(\C,\D)$ is called the \emph{projective model structure}. 
\end{theorem}

\begin{remark}
We often refer to the weak equivalences in this model structure as \emph{levelwise} weak equivalences, and similarly for the fibrations.  Observe that the fibrant objects in this model structure are those functors $F \colon \C \rightarrow \D$ such that $FA$ is fibrant in $\D$ for every object $A$ of $\C$.
\end{remark}

\begin{lemma} 
When it exists, the projective model structure on $\Fun(\C, \D)$ has the structure of a $\V$-model category.  
\end{lemma}

\begin{proof}
We have proved axiom (MC6) in Propositions \ref{Fun is tensored} and \ref{Fun is cotensored}, so it remains to show that axiom (MC7) is satisfied. The fibrations and weak equivalences in $\Fun(\C, \D)$ are the levelwise fibrations and weak equivalences, respectively, and (MC7) holds in $\D$ by assumption. It follows that for any fibration $p \colon F \rightarrow G$ in $\Fun(\C, \D)$ and cofibration $i \colon V \rightarrow W$ in $\V$, the pullback corner map 
\[ \begin{tikzpicture}[node distance=3.5cm, auto]

\node (A) {$F^V \times_{G^V} G^W$};
\node (B) [right of=A] {$F^V$};
\node (C) [node distance=1.8cm, below of=A] {$G^W$};
\node (D) [right of=C] {$G^V$};
\node (E') [node distance=1.2cm, above of=A, left of=A] {};
\node (E) [node distance=0.8cm, left of=E'] {$F^W$};

\draw[->] (A) to node {} (B);
\draw[->] (C) to node [swap] {} (D);
\draw[->] (A) to node [swap] {} (C);
\draw[->] (B) to node {} (D);
\draw[->, bend right] (E) to node {} (C);
\draw[->, bend left=16] (E) to node {} (B);
\draw[->, dashed] (E) to node [swap]  {} (A);

\end{tikzpicture} \] 
is a fibration, and is a weak equivalence if either $i$ or $p$ is.
\end{proof}

\section{Homotopy limits, homotopy colimits, and the evaluated cotensor in simplicial model categories} \label{sectioncotensor} 

With this section, we begin our transition to functor calculus applications.  A key step in the process of defining functor calculus model structures as in \cite{br} is the redefinition of polynomial approximations in terms of the evaluated cotensor of Definition \ref{cotensor}.  We use this section to prove a result showing how this evaluated cotensor interacts with homotopy limits and colimits in simplicial model categories.  This result will be used in the next section to show that these redefined polynomial approximations are equivalent to those obtained via the original definitions.  

We begin this section by first looking at limits and colimits.  

\begin{lemma} \label{lim-colim lemma} 
Let $\D$ be a complete $\V$-category that is tensored and cotensored over $\V$, and let $\V$ be cocomplete. For any  object $D$ of $\D$, small category $\I$, and functor $\I \rightarrow \V$, there is a natural isomorphism 
\[ D^{\underset{i}{\colim}{V_i}} \cong \lim_i (D^{V_i}). \]
\end{lemma}

\begin{proof} 
Since all colimits can be built from coproducts and coequalizers \cite[3.4.11]{riehl}, it is enough to prove the lemma for these two kinds of colimits. For coproducts, we have to show that given objects $V$ and $W$ in $\V$, there is a natural isomorphism $D^{V \amalg W} \cong D^V \times D^W$. We do so by showing that for each object $E$ of $\D$ there is a natural isomorphism of sets
\[ \Hom_{\D}(E, D^{V\amalg W}) \cong \Hom_{\D}(E, D^V \times D^W). \]
A morphism $E \rightarrow D^{V \amalg W}$ is equivalently given by its adjoint morphism $E \otimes (V \amalg W) \rightarrow D$.  Since the functor $E \otimes - $ is a left adjoint (Proposition \ref{tensor-cotensor adjunctions}), it preserves colimits, so $E \otimes (V \amalg W) \cong (E \otimes V) \amalg (E \otimes W)$. Thus any morphism $E \rightarrow D^{V \amalg W}$ is equivalently given by a pair of morphisms
\[ (E \otimes V \rightarrow D, E \otimes W \rightarrow D) \]
or, via adjunction (3) from Proposition \ref{tensor-cotensor adjunctions}, a pair
\[ (E \rightarrow D^V, E \rightarrow D^W). \]
By the universal property of products, this pair is equivalent to a morphism $E \rightarrow D^V \times D^W$.

For coequalizers, we want to show that for any object $E$ of $\D$ and parallel pair $U \rightrightarrows V$ of morphisms in $\V$ there is a natural isomorphism
\[ \Hom_{\D}(E, D^W) \cong \Hom_{\D}(E, Z), \]
where $W$ is the coequalizer
\[ U \rightrightarrows V \rightarrow W \]
in $\mathcal{V}$, and $Z$ is the equalizer
\[ Z \rightarrow D^V \rightrightarrows D^U. \]
A morphism $E \rightarrow D^W$ is equivalently given by a morphism $E \otimes W \rightarrow D$, and since $E \otimes -$ preserves coequalizers, it is equivalent to a morphism $E \otimes V \rightarrow D$ such that the diagram
\[ E \otimes U \rightrightarrows E \otimes V \rightarrow D \]
commutes.  Applying adjunction (3) of Proposition \ref{tensor-cotensor adjunctions} gives the diagram
\[ E \rightarrow D^V \rightrightarrows D^U \]
which, by the universal property of equalizers, is equivalent to a morphism $E \rightarrow Z$.
\end{proof}

We next consider how the evaluated cotensor of Definition \ref{cotensor} interacts with limits and colimits.

\begin{lemma} \label{l:homlim}  
Let $\D$ be a complete  $\V$-category that is tensored and cotensored over $\V$, and let $\V$ be cocomplete.
Let $\I$ be a small category and let $\mathbf{X}$ be a functor $\mathbf{X} \colon \I \rightarrow \Fun(\C,\V).$  Then for any $\V$-functor $F \colon \C \rightarrow \D$, there is a natural isomorphism 
\[ F^{\underset{i}{\colim} \mathbf{X}i} \cong \lim_{i} F^{\mathbf{X}i}. \] 
\end{lemma}

\begin{proof}
Using the definition of the evaluated cotensor, we compute
\[ \begin{aligned}
    F^{\underset{i}{\colim} \mathbf{X}_i} & :=\int_{C} FC^{\underset{i}{\colim} \mathbf{X}_iC} \\
    & \cong \int_{C}\lim_{i} FC^{\mathbf{X}_iC} \\
    & \cong \lim_{i} \int_{C} FC^{\mathbf{X}_iC} \\
    & := \lim_{i} F^{\mathbf{X}i}.
\end{aligned} \]
The first isomorphism holds by Lemma \ref{lim-colim lemma}, and the second isomorphism holds because limits commute \cite[3.8.1]{riehl}. 
\end{proof}

For the remainder of this section, we let $\V=\s$, the category of simplicial sets, and hence work exclusively with functors between simplicial model categories.  

Since we are working in a homotopy-theoretic setting, we want to work with homotopy limits and colimits, so that our constructions are homotopy invariant.  We start by recalling the standard constructions of homotopy limits and colimits in simplicial model categories as found in \cite[18.1.2, 18.1.8]{hirschhorn}.  We note that these constructions include an implicit assumption that the diagrams are objectwise fibrant in the case of homotopy limits and cofibrant for homotopy colimits. 

\begin{definition} \label{d:hocolim} 
Let $\M$ be a simplicial model category and $\I$ a small category. If $\mathbf{X}$ is an $\I$-diagram in $\M$, then the \emph{homotopy colimit of $\mathbf{X}$}, denoted by $\underset{i}{\hocolim} \mathbf{X}_i$, is the coequalizer of the maps 
\[ \coprod_{\sigma \colon a \rightarrow a'} \mathbf{X}_a \otimes B(a'\downarrow \I)^{\op} \underset{\psi}{\overset{\phi}{\rightrightarrows}} \coprod_{a} \mathbf{X}_a \otimes B(a \downarrow \I)^{\op}, \] 
where $\phi$ is defined on the summand corresponding to $\sigma \colon a \rightarrow a'$ to be the composite of the map 
\[ \sigma_* \otimes \Id_{B(a'\downarrow \I)} \colon \mathbf{X}_a \otimes B(a'\downarrow \I)^{\op} \rightarrow  \mathbf{X}_{a'} \otimes B(a'\downarrow \I)^{\op} \] 
with the natural injection into the coproduct, and $\psi$ is given by the composite of the map 
\[ \Id_{\mathbf{X}_a}\otimes B(\sigma^*) \colon \mathbf{X}_a \otimes B(a \downarrow \I)^{\op} \rightarrow \mathbf{X}_{a'} \otimes B(a'\downarrow \C)^{\op} \] 
with the natural injection into the coproduct.  Here, $\sigma^*$ denotes the functor induced by precomposition with $\sigma$, $B$ denotes the nerve of a category, and $a\downarrow \I$  denotes the category of objects of $\I$ under $a$; see \cite[11.8.3, 14.1.1]{hirschhorn} for details.
\end{definition}

\begin{definition} 
Let $\M$ be a simplicial model category and let $\I$ be a small category. If $\mathbf{X}$ is an $\I$-diagram in $\M$, then the \emph{homotopy limit} of $\mathbf{X}$, denoted by $\underset{i}{\holim} \mathbf{X}_i$, is defined as the equalizer of the maps 
\[ \prod_{a} \mathbf{X}_a^{B(\I \downarrow a)} \underset{\psi}{\overset{\phi}{\rightrightarrows}} \prod_{(\sigma \colon a \rightarrow a')} \mathbf{X}_{a'}^{B(\I \downarrow a)}, \] 
where the projection of the map $\phi$ on the factor $\sigma \colon a \rightarrow a'$ is the composite of a natural projection from the product with the map
\[ \sigma_*^{\Id_{B(\I \downarrow a)}} \colon \mathbf{X}_a^{B(\I \downarrow a)} \rightarrow \mathbf{X}_{a'}^{B(\I \downarrow a)}, \]
and the projection of the map $\psi$ is given by the composite of a natural projection from the product with the map 
\[ (\Id_{\mathbf{X}_{a'}})^{B(\sigma_*)} \colon \mathbf{X}_{a'}^{B(\I \downarrow a)} \rightarrow \mathbf{X}_{a'}^{B(\I \downarrow a)}. \] 
Here, $\sigma_*$ denotes the functor induced by postcomposition with $\sigma$, and $\I \downarrow a$ denotes the category of objects in $\I$ over $a$.
\end{definition}

Using these models for homotopy limits and colimits, we establish the desired result for the evaluated cotensor.

\begin{proposition} \label{p:homholim}
Let $\I$ be a small category and $\mathbf{X}$ an $\I$-diagram in $\Fun(\C,\s)$.  Let $\D$ be a cofibrantly generated simplicial model category satisfying the conditions of Theorem \ref{projmodel}, so that $\Fun(\C,\D)$ has the structure of a simplicial model category. Then for any $F$ in $\Fun(\C,\D)$, 
\[ F^{\underset{i}{\hocolim} \mathbf{X}_i} \cong \underset{i}{\holim} F^{\mathbf{X}_i}. \]
\end{proposition}

\begin{proof}
By the definition of homotopy colimit and Lemma \ref{l:homlim}, we have 
\[ \begin{aligned}
    F^{\underset{i}{\hocolim} \mathbf{X}_i} & = F^{\colim \left(\coprod_{i \rightarrow i'} \mathbf{X}_i\otimes B(i'\downarrow \I)^{\op} \rightrightarrows \coprod _i \mathbf{X}_i\otimes B(i \downarrow \I)^{\op} \right)} \\
    & \cong \lim \left (\prod_i F^{\mathbf{X}_i \otimes B(i \downarrow \I)^{\op}} \rightrightarrows \prod_{i\rightarrow i'} F^{\mathbf{X}_i\otimes B(i'\downarrow \I)^{\op}} \right).
\end{aligned} \]

The definition of the evaluated cotensor, the associativity of tensoring and cotensoring \cite[9.1.11]{hirschhorn}, and Proposition \ref{l:homcotensor} yield
\[ \begin{aligned}
F^{\mathbf{X}_i\otimes B(i'\downarrow \I)^{\op}} & = \int_{C} FC^{\mathbf{X}_i C \otimes B(i'\downarrow \I)^{\op}} \\
& \cong \int_{C} \left( FC^{\mathbf{X}_iC} \right)^{B(i' \downarrow \I)^{\op}} \\
& \cong \left( \int_{C} FC^{\mathbf{X}_iC}  \right)^{B(i' \downarrow \I)^{\op}} \\
& = \left( F^{\mathbf{X}_i} \right)^{B(i'\downarrow \I)^{\op}}.
\end{aligned} \]
Combining this computation with the previous equivalence, we obtain
\[ \begin{aligned}
    F^{\underset{i}{\hocolim} \mathbf{X}_i} & = \lim \left (\prod_i F^{\mathbf{X}_i\otimes B(i \downarrow \I)^{\op}} \rightrightarrows \prod_{i \rightarrow i'} F^{\mathbf{X}_i \otimes B(i'\downarrow \I)^{\op}} \right) \\
    & \cong \lim \left( \prod_i \left( F^{\mathbf{X}_i} \right) ^{B(\I^{\op}\downarrow i)} \rightrightarrows \prod_{i \rightarrow i'} \left(F^{\mathbf{X}_i} \right)^{B(\I^{\op} \downarrow i')} \right )\\
    & = \holim_i F^{\mathbf{X}_i}, 
\end{aligned} \]
where the middle isomorphism follows from \cite[11.8.7]{hirschhorn}.
\end{proof}

\section{An application to functor calculus}
\label{sectionfunctorcalc}

In this section, we establish results that allow one to replace the standard constructions of some functor calculi, particularly those of \cite{goodwillie3} and \cite{bjm}, with enriched constructions that can be used to produce cofibrantly generated model structures for these functor calculi.  In the case of the Goodwillie calculus, Biedermann and R\"ondigs have shown in \cite{br} how to use these enriched constructions to build cofibrantly generated $n$-excisive model structures. In the sequel to this paper \cite{sequel}, we will provide a general result that recovers this structure and produces similar structures for other functor calculi.

After stating a corollary to Proposition \ref{p:homholim}, we divide the remainder of this section into three parts.  In the first subsection, we apply Corollary \ref{c:reprep} to redefine the building blocks $T_n$ of Goodwillie's $n$-excisive approximations to functors in terms of the evaluated cotensor of Definition \ref{Fend} and representable functors of Definition \ref{representable}.  In the second subsection, we make a similar redefinition of the discrete degree $n$ approximations $\bot_n$ of Bauer, Johnson, and McCarthy.  Finally, in the third subsection, we show that these redefined building blocks define simplicial endofunctors on $\Fun(\C,\D)$.

Throughout this section, we work with simplicial enrichments.  The categories $\C$ and $\D$ are assumed to be simplicial categories, with $\D$ a simplicial model category.  While we would like to ask that $\C$ also be a simplicial model category, we need some smallness conditions so that the category $\Fun(\C,\D)$ of simplicial functors can be given the projective model structure as above.  Yet, we also want to be able to apply certain model category techniques, such as homotopy pushouts, in $\C$.

Thus, we ask that $\C$ be a small simplicial subcategory of a simplicial model category that has a final object $\ast_\C$.  In addition, we want $\C$ to be tensored  in finite simplicial sets; to satisfy this condition we may allow $\C$ to be \emph{essentially small}, namely, equivalent to a small simplicial category.  Each of our two situations requires some further assumptions, but we defer mention of them to the appropriate subsections below.

The main results of the following two subsections are consequences of the following corollary to Proposition \ref{p:homholim}. 

\begin{corollary} \label{c:reprep}  
Let $\I$ be a small category and $\mathbf{X}$ be an $\I$-diagram in $\C$.  Then for a simplicial functor $F$, there is a natural isomorphism 
\[ \underset{i}{\holim} F\mathbf{X}_i \cong F^{\underset{i}{\hocolim} R^{\mathbf{X}_i}}. \] 
\end{corollary}

\begin{proof}
By Proposition \ref{p:homholim}, followed by the enriched Yoneda Lemma (Lemma \ref{Yoneda}), we have
\[ \begin{aligned}
F^{\underset{i}{\hocolim} R^{\mathbf{X}_i}} & \cong \underset{i}{\holim} F^{R^{\mathbf{X}_i}} \\
& = \underset{i}{\holim} \int_{C}FC^{R^{\mathbf{X}_i}C} \\
& \cong \underset{i}{\holim} F\mathbf{X}_i. 
\end{aligned} \] 
\end{proof}

We apply this result to two examples in the next two subsections.

\subsection{A simplicial representable replacement for Goodwillie calculus} \label{tnrep}

In the Goodwillie calculus, the $n$-\emph{excisive approximation} to a functor $F$ is a functor $P_nF$, which is defined as the homotopy colimit of a sequence of functors:
\[ P_nF:= \hocolim_k \left(F \rightarrow T_nF \rightarrow T_n^2F \rightarrow \dots \rightarrow T_n^kF \rightarrow \dots \right). \]
In this section, we are interested in the functor $T_n$, and in particular we want to show how, for simplicial functors, it can be defined in terms of certain representable functors, using Corollary \ref{c:reprep}.

Throughout this subsection, we assume the conditions  on the categories $\C$ and $\D$ described in the introduction to the section; we additionally assume that $\C$ is closed under finite homotopy pushouts.

To define $T_nF$ as Goodwillie did, we start with some notation.  For any $n \geq 1$, let $\mathbf n=\{1, \ldots, n\}$.  We think of the power set $\P(\mathbf n)$ as a poset whose elements are subsets of $\mathbf n$ and whose ordering is given by set inclusion.  We can thus regard $\P(\mathbf n)$ as a category and will make use of two of its subcategories. Let $\P_{\leq 1}(\mathbf{n})$ be the full subcategory whose objects are the subsets of $\mathbf{n}$ of cardinality $0$ and $1$ and
$\P_0(\mathbf{n})$ be the full subcategory of $\P(\mathbf n)$ whose objects are the nonempty subsets.  
For instance, the diagrams below represent $\P(\mathbf 3)$, $\P_0(\mathbf 3)$ and $\P_{\leq 1}(\mathbf{3})$, respectively:

\begin{center}
\begin{tikzpicture}[node distance=2.5cm, auto, scale=1]

\node (A) {$\emptyset$};
\node (B) [right of=A] {$\{1\}$};
\node (C) [node distance=2.3cm, below of=A] {$\{2\}$};
\node (D) [node distance=2.3cm, below of=B] {$\{1,2\}$};

\node (X) [node distance=1cm, above of=A] {};

\node (E) [node distance=1.5cm, right of=X] {$\{3\}$};
\node (F) [right of=E] {$\{1,3\}$};
\node (G) [node distance=2.3cm,below of=E] {$\{2,3\}$};
\node (H) [node distance=2.3cm,below of=F] {$\{1,2,3\}$};

\draw[->] (A) to node {} (B);
\draw[->] (A) to node {} (C);
\draw[->] (C) to node {} (D);
\draw[->] (B) to node {} (D);

\draw[->] (A) to node {} (E);
\draw[->] (B) to node {} (F);
\draw[->] (E) to node {} (F);
\draw[->] (F) to node {} (H);
\draw[->] (D) to node {} (H);

\draw[->, dashed] (G) to node {} (H);
\draw[->, dashed] (C) to node {} (G);
\draw[->, dashed] (E) to node {} (G);

\node (A') [node distance=4.8cm, right of=A] {};
\node (B') [right of=A'] {$\{1\}$};
\node (C') [node distance=2.3cm, below of=A'] {$\{2\}$};
\node (D') [node distance=2.3cm, below of=B'] {$\{1,2\}$};

\node (X') [node distance=1cm, above of=A'] {};

\node (E') [node distance=1.5cm, right of=X'] {$\{3\}$};
\node (F') [right of=E'] {$\{1,3\}$};
\node (G') [node distance=2.3cm,below of=E'] {$\{2,3\}$};
\node (H') [node distance=2.3cm,below of=F'] {$\{1,2,3\}$};

\draw[->] (C') to node {} (D');
\draw[->] (B') to node {} (D');

\draw[->] (B') to node {} (F');
\draw[->] (E') to node {} (F');
\draw[->] (F') to node {} (H');
\draw[->] (D') to node {} (H');

\draw[->, dashed] (G') to node {} (H');
\draw[->] (C') to node {} (G');
\draw[->] (E') to node {} (G');

\node (A'') [node distance=5.5cm, right of=A']{$\emptyset$};
\node (B'') [right of=A''] {$\{1\}$.};
\node (C'') [node distance=2.3cm, below of=A''] {$\{2\}$};
\node (X'') [node distance=1cm, above of=A''] {};
\node (E'') [node distance=1.5cm, right of=X''] {$\{3\}$};

\draw[->] (A'') to node {} (B'');
\draw[->] (A'') to node {} (C'');
\draw[->] (A'') to node {} (E'');

\end{tikzpicture}
\end{center}

For an object $A$ in $\C$ and a finite set $U$, we define the \emph{fiberwise join} $A\ast U$ in $\C$ to be the homotopy colimit of the  $\P_{\leq 1}(\mathbf{n})$-diagram in $\C$ that assigns to $\varnothing$ the object $A$ and assigns to each one-element set $\{i\}$ the final object $\ast_\C$.  So, for $n=1$, the fiberwise join is just the simplicial cone on $A$, and in general, $A\ast U$ is $n$ copies of the simplicial cone on $A$ glued along $A$. As an example, for $|U| = 4$, $A\ast U$ is the homotopy colimit of the diagram below: 
\begin{center}
\begin{tikzpicture}[node distance=2.5cm, auto, scale=1]

\node (A) {$A$};
\node (B) [right of=A] {$\ast_\C$.};
\node (C) [left of=A] {$\ast_\C$};
\node (X) [node distance=1cm, below of=A] {};
\node (E) [node distance=1.5cm, right of=X] {$\ast_\C$};
\node (F) [node distance=1.5cm, left of=X] {$\ast_\C$};

\draw[->] (A) to node {} (B);
\draw[->] (A) to node {} (C);
\draw[->] (A) to node {} (E);
\draw[->] (A) to node {} (F);

\end{tikzpicture}
\end{center}

The following definition is due to Goodwillie in the case where $\C$ and $\D$ are each either the category of spaces or the category of spectra.  

\begin{definition} \cite{goodwillie3} \label{d:tn}
Given a functor $F \colon \C \rightarrow \D$, we define the functor $T_nF \colon \C \rightarrow \D$ by 
\[ T_nF(A) = \underset{U \in \mathcal{P}_0(\mathbf{n+1})}{\holim} F(A \ast U). \]
\end{definition}

Following Biedermann and R\"ondigs \cite{br}, we have the following version of $T_nF$ for a simplicial functor $F$.

\begin{definition} \label{rept}
Given a simplicial functor $F \colon \C \rightarrow \D$, we define the functor $T_n^{R}F \colon \C \rightarrow \D$ by 
\[ T_n^{R}F(A) =F^{\underset{U}{\hocolim}{R^{A \ast U}}} = \int_{C} FC^{ \underset{U}{\hocolim} R^{A \ast U}C} = \int_{C} FC^{ \underset{U}{\hocolim} \Map_{\C}(A \ast U,C)}, \]
where the homotopy colimits are taken over all objects $U$ of $\P_0(\mathbf{n+1})$. 
\end{definition}

The next proposition is an immediate consequence of Corollary \ref{c:reprep}.

\begin{proposition}\label{p:brrep}
If $F \colon \C \rightarrow \D$ is a simplicial functor, then the functors $T_nF$ and $T_n^{R}F$ are isomorphic.
\end{proposition}

\begin{proof}
For an object $A$ in $\C$, applying Corollary \ref{c:reprep} to $T_nF(A)$ gives us
\[ \begin{aligned}
    T_nF(A) & = \underset{U}\holim F(A\ast U)\\
    &\cong F^{\underset{U}\hocolim R^{A\ast U}}\\
    &=T^R_nF(A).
\end{aligned} \]
\end{proof}

\begin{remark}
To define their $n$-excisive model structures, Biedermann and R\"ondigs define $T^R_nF(A)$ as $F^{A_n}$ where $A_n$ is simplicially homotopy equivalent to $\underset{U}{\hocolim} R^{A \ast U}$ in $\Fun(\C,\s).$  Since simplicial homotopy equivalences are preserved by simplicial functors, it follows that $F^{A_n}\simeq T_n^RF(A)\cong T_nF(A)$.
\end{remark}

\subsection{A simplicial representable replacement for discrete calculus} \label{discreterep}

The discrete functor calculus of Bauer, Johnson, and McCarthy \cite{bjm} is an adaptation of the abelian functor calculus of Johnson and McCarthy \cite{jm}  to simplicial model categories.  Like Goodwillie calculus, it associates a ``degree $n$" polynomial approximation $\Gamma_nF$ to a functor $F$.  The notion of ``degree $n$" in this case is weaker than that of Goodwillie, and closer in spirit to the notion of degree $n$ for polynomial functions.  

In this subsection, we make two assumptions on the categories $\C$ and $\D$, in addition to those specified at the beginning of the section.  We assume  that $\C$  is closed under finite coproducts, and that $\D$ is {\emph{pointed}} with initial and terminal object $\star$. (To build polynomial approximations in \cite{bjm}, we also require that $\D$ be stable, but for the results in this paper the stability condition is not necessary.)   

Recall that a comonad $(\bot, \Delta, \varepsilon)$ acting on a category $\A$ consists of an endofunctor $\bot \colon \A \rightarrow \A$ together with natural transformations $\Delta \colon \bot\rightarrow \bot\bot$ and $\varepsilon \colon \bot\rightarrow \Id_\A$ satisfying certain identities.  For an object $A$ in $\A$, there is an associated simplicial object 
\[ [k] \mapsto \bot^{k+1}A \]
whose face and degeneracy maps are defined using the natural transformations $\varepsilon$ and $\Delta$.  (See \cite[\S 8.6]{weibel} for more details, noting that the author uses the term ``cotriple" for what we are calling a ``comonad" here.)

The functor $\Gamma_nF$ is defined in terms of a comonad $\bot_{n+1}$ that acts on the category of functors from $\C$ to $\D.$  More explicitly, it is the homotopy cofiber given by  
\[ \Gamma_nF:=\hocofiber(\vert \bot_{n+1}^{\ast+1}F\vert \rightarrow F) \]
where $\vert \bot_{n+1}^{\ast+1}F\vert$ is the realization of the standard simplicial object associated to the comonad $\bot_{n+1}$ acting on $F.$ 

We first review the definition of $\bot_n$. We again make use of the power set $\P(\mathbf {n})$ of the set $\mathbf{n}=\{1, \ldots, n\}$, regarded as a category, and the category $\P_0(\mathbf{2})$, the full subcategory of $\P(\mathbf{2})$ whose objects are the nonempty subsets of $\{1,2\}.$ 

For an object $A$ in $\C$ and a functor $F \colon \C\rightarrow \D$, let $F^{\mathbf{n}}(A,-)$ be the $\P(\mathbf{n})$-diagram that assigns to $U \subseteq \mathbf{n}$ the object
\[ F^{\mathbf{n}}(A,U):=F \left(\coprod _{i\in \mathbf{n}} A_i(U) \right), \]
where
\begin{equation} \label{e:Ai}
A_i(U):= \begin{cases}
A & i\notin U, \\
\ast_\mathcal{C} & i\in U.
\end{cases}
\end{equation} 
The value of the functor $\bot_{n} F \colon \C\rightarrow \D$ at the object $A$ is defined as an iterated homotopy fiber of the diagram $F^{\mathbf n}(A, -)$, as we now explain.

For any $F \colon \C \rightarrow \D$ as above and any object $A$ of $\C$, let $F^{\P_0(\mathbf{2})}_n \colon \C \times {\P_0(\mathbf{2})}^{\times n} \rightarrow \D$ be given by 
\[ F^{\P_0(\mathbf{2})}_n(A;(S_1, \dots, S_n)):=
\begin{cases}
F^{\mathbf{n}}(A,\varphi(S_1, \dots, S_n)) & S_i\neq \{2\} \text{ for all } i, \\
\star & \text{otherwise,}
\end{cases} \]
where 
\begin{equation} \label{e:phi}
\varphi(S_1, \dots, S_n)=\{i\ |\ S_i=\{1,2\}\}.
\end{equation}

\begin{example}
To help make sense of this definition, we consider the example of $n=2$.  Then ${\P_0(\mathbf{2})} \times {\P_0(\mathbf{2})}$ can be depicted as
\[ \xymatrix{ (\{1\}, \{1\}) \ar[d] \ar[r] & (\{1\}, \{1,2\}) \ar[d] & (\{1\}, \{2\}) \ar[l] \ar[d] \\
(\{1,2\}, \{1\}) \ar[r] & (\{1, 2\}, \{1, 2\}) & (\{1,2\},\{2\}) \ar[l] \\
(\{2\},\{1\})\ar[u] \ar[r] & (\{2\}, \{1,2\}) \ar[u] & (\{2\},\{2\}). \ar[l] \ar[u]} \]
\vskip .2 in 
The relevant values of $\varphi$ in this case are
\[ \begin{aligned}
\varphi (\{1\}, \{1\}) & = \varnothing \\
\varphi (\{1\}, \{1,2\}) & = \{2\} \\
\varphi(\{1,2\}, \{1\}) & = \{1\} \\
\varphi(\{1,2\},\{1,2\})&=\{1,2\}.
\end{aligned} \]

Then $F_2^{\P_0(\mathbf{2})}(A,-)$ is given by the diagram 
\[ \xymatrix{F^2(A,\varnothing) \ar[r] \ar[d] & F^2(A,\{2\}) \ar[d] & \star\ar[l] \ar[d] \\
F^2(A, \{1\}) \ar[r] & F^2(A, \{1,2\}) & \star \ar[l] \\
\star \ar[u] \ar[r] & \star \ar[u] & \star, \ar[u] \ar[l]} \]
which can be rewritten as
\[ \xymatrix{F(A \amalg A) \ar[r] \ar[d]& F(A \amalg \ast_\C) \ar[d] & \star \ar[l] \ar[d] \\
F(\ast_\C \amalg A) \ar[r] & F(\ast_\C \amalg \ast_\C) & \star \ar[l] \\
\star \ar[u] \ar[r] & \star \ar[u] & \star. \ar[u] \ar[l]} \]
\end{example}

\begin{definition}
For any functor $F \colon \C \rightarrow \D$ and any object $A$ of $\C$, define
\[ \bot_{n} F(A):=\holim_{(S_1, \dots, S_n)} F^{\P_0(\mathbf{2})}_n \left(A; (S_1, \dots, S_n) \right), \]
where the homotopy limit is taken over the category ${\P_0(\mathbf{2})}^{\times n}$.
\end{definition}

Observe that $\bot_nF(A)$ is equivalent to the iterated homotopy fiber of the $\P(\mathbf{n})$-diagram that assigns to the set $U\subseteq \mathbf{n}$ the object $F^{\mathbf n}(A,U).$  For example, $\bot_2F(A)$ is equivalent to the object obtained by taking the homotopy fibers of the vertical morphisms in 
\[ \xymatrix{F(A\coprod A)\ar[r]\ar[d]&F(\A\coprod \ast_\C)\ar[d]\\
F(\ast_\C\coprod A)\ar[r]&F(\ast_\C\coprod \ast_\C)} \]
and then taking the homotopy fiber of the induced map between the two vertical homotopy fibers. See \cite[\S 3.1]{bjm} for details. 

Our next step is to define a representable version of  $\bot_nF$  in a manner similar to that used for $T_nF$  in Definition \ref{rept} and Proposition \ref{p:brrep}. To do so, we rewrite $F_n^{\P_0(\mathbf{2})}(-;(S_1,\dots, S_n))$ as a functor applied to objects in a single category $\C_\bot$,   obtained from $\C$ by adjoining an initial object $\bot$. In other words, $\C_\bot$ is the category whose objects are the objects of $\C$ together with one additional object $\bot$, and whose morphisms are the morphisms of $\C$ together with the identity morphism on $\bot$ and a unique morphism $\bot \rightarrow A$ for each object $A$ in $\C$.

Given a functor $F \colon \C \rightarrow \D$, we can extend it to a functor $F \colon \C_\bot \rightarrow \D$ 
defined on objects by 
\[ F(A):=
\begin{cases}
F(A) & A \in \ob(\C)\\
\star & A=\bot
\end{cases} \]
and on morphisms by 
\[ F(f \colon A \rightarrow B) =
\begin{cases}
F(f) & f \text{ is a morphism in }\C \\
\Id_\star & f=\Id_\bot \\
\star \rightarrow F(B) & A=\bot,
\end{cases} \]
where in the last case the arrow is the unique one from the zero object $\star$ of $\D$.  It is straightforward to confirm that $\Fun(\C_\bot, \D)$ is also a simplicial model category

For an object $A$ in $\C$ and an object $S=(S_1, \dots, S_n)$ of ${\P_0(\mathbf{2})}^{\times n}$, using \eqref{e:Ai} and \eqref{e:phi}, we define 
\begin{equation} \label{sqcupbotdefn}
A \sqcup S:=
\begin{cases}
\bot & S_i=\{2\}\ \text{ for some } i,\\
\underset{i\in \mathbf{n}}{\coprod} A_i(\varphi(S)) & \text{otherwise}.
\end{cases}
\end{equation}
Note that 
\[
F_n^{{\P_0(\mathbf{2})}}(A;S)=F(A\sqcup S),
\] so  we can rewrite $\bot_nF(A)$ as
\[ \bot_nF(A)=\underset{S}{\holim}F(A \sqcup S). \]

Applying Corollary \ref{c:reprep}, we see that
\[
\bot_nF(A)\cong F^{\underset{S}{\hocolim}R^{A\sqcup S}}=\int_{C}FC^{\underset{S}{\hocolim}R^{A\sqcup S}(C)}.
\]
We summarize these ideas in the definition of the representable version of $\bot_nF$ and subsequent proposition.

\begin{definition} \label{d:repbot}
Let $F \colon \C \rightarrow \D$ be a simplicial functor.  We define, for an object $A$ in $\C$,
\begin{equation}
    \bot_{n}^R F(A):=F^{\underset{S}{\hocolim}R^{A\sqcup S}}=\int_{C}FC^{\underset{S}{\hocolim}R^{A\sqcup S}(C)},
\end{equation}
where each homotopy colimit is taken over the category ${\P_0(\mathbf{2})}^{\times n}$.
\end{definition}

\begin{prop}\label{p:botdefagree}
Let $F \colon \C \rightarrow \D$ be a simplicial functor.  Then $\bot_nF$ is isomorphic to $\bot_n^RF$. 
\end{prop}

\subsection{Simplicial functors} \label{tnsimp}

Finally, in this subsection we prove that $\bot_n^R$ and $T_n^R$ define simplicial functors.  

\begin{proposition} \label{trsimp}
Let $\C$ be a small simplicial subcategory of a simplicial model category that is closed under finite homotopy pushouts and tensored in finite simplicial sets, and let $\D$ be  a simplicial model category.  If $F \colon \C \rightarrow \D$ is a simplicial functor, then so is $T^R_nF$.
\end{proposition}

\begin{lemma} \label{joinusimp}
For a fixed finite set $U$, the join functor $- \ast U \colon \C \rightarrow \C$ is a simplicial functor.
\end{lemma}

\begin{proof}
Since $\C$ is tensored in finite simplicial sets, we have, for any $n \geq 0$ and any objects $A$ and $B$ in $\C$, a map 
\[ \Hom_\C(A\otimes\Delta[n],B) \rightarrow \Hom_\C \left( (A \otimes \Delta[n])\ast U, B\ast U \right). \]
As a left adjoint, tensoring with $\Delta[n]$ commutes with the join construction to yield a map
\[ \Hom_\C(A\otimes\Delta[n],B) \rightarrow \Hom_\C \left( (A\ast U) \otimes \Delta[n], B\ast U \right). \]
By a basic property of mapping spaces in simplicial categories  \cite[II.2.2]{gj}, we can consider this map instead as
\[ \Map_\C(A,B)_n \rightarrow \Map_\C(A\ast U, B\ast U)_n. \]
As a result, we have a map of simplicial sets
\[ \Map_\C(A,B) \rightarrow \Map_\C(A\ast U, B\ast U) \] 
and can conclude that $-\ast U$ is a simplicial functor.
\end{proof}

\begin{proof}[Proof of Proposition \ref{trsimp}]
We want to show that $T_n^RF \colon \C \rightarrow \D$ is a simplicial functor, so for any objects $A$ and $B$ of $\C$, we need to define compatible maps of simplicial sets
\[ \Map_\C(A,B) \rightarrow \Map_\D(T^R_nFA, T^R_nFB), \]
namely, applying the definition of $T_n^R$, maps
\[ \Map_\C(A,B) \rightarrow \Map_\D \left( \int_C FC^{\underset{U}{\hocolim} \Map_\C(A \ast U, C)}, \int_C FC^{\underset{U}{\hocolim} \Map_\C(B \ast U, C)} \right). \]

We apply the co-Yoneda embedding defined following Proposition \ref{l:homcotensor} to obtain maps
\[ \Map_\C(A,B) \rightarrow \Map_\C(A \ast U, B \ast U) \rightarrow \Map_\C(A \ast U, C)^{\Map_\C(B \ast U, C)}, \]
where $C$ is any other object of $\C$. Note here we have also used the fact that $-\ast U$ is simplicial from Lemma \ref{joinusimp}.  Using the isomorphism $\Map_\C(A,B) \cong I \otimes \Map_\C(A,B)$ and the unit map $I \rightarrow \Map_\D(FC,FC)$, we obtain a map
\[ \Map_\C(A,B) \rightarrow \Map_\D(FC,FC) \otimes \Map_\C(A \ast U, C)^{\Map_\C(B \ast U, C)}. \]

Now, applying the assignment in Definition \ref{cotensor def} to $D=E=FC$, $V=\Map_\C(A \ast U, C)$, and $W=\Map_\C(B \ast U, C)$, produces a map
\[ \Map_\D(FC,FC) \otimes \Map_\C(A \ast U, C)^{\Map_\C(B \ast U, C)} \rightarrow \Map_\D(FC^{\Map_\C(A \ast U,C)}, FC^{\Map_\C(B \ast U, C)}). \]
Precomposing with the previous map, we thus have a map
\[ \Map_\C(A,B) \rightarrow \Map_\D(FC^{\Map_\C(A \ast U,C)}, FC^{\Map_\C(B \ast U, C)}). \]

We can now apply a homotopy colimit over all $U$ in the cotensors to get a map
\[ \Map_\C(A,B) \rightarrow \Map_\D \left( FC^{\underset{U}{\hocolim}\Map_\C(A \ast U,C)}, FC^{\underset{U}{\hocolim} \Map_\C(B \ast U, C)} \right), \]
and then taking an end over all objects $C$ of $\C$ gives the desired map.
\end{proof}

\begin{proposition} 
The functor 
\[ T_n^{R} \colon \Fun(\C,\D) \rightarrow \Fun(\C, \D) \]
given by
\[ F \mapsto T_n^{R}F \]
is a simplicial functor.
\end{proposition}

\begin{proof} 
We need to define compatible morphisms of simplicial sets
\[ \Map_{\Fun(\C,\D)}(F,G) \rightarrow \Map_{\Fun(\C,\D)}(T^R_nF, T^R_nG) \]
for any simplicial functors $F, G \colon \C \rightarrow \D$.  Applying the definition of $T^R_n$, we need to define a map
\[ \Map_{\Fun(\C,\D)}(F,G) \rightarrow \Map_{\Fun(\C,\D)}\left(\int_C FC^{\underset{U}{\hocolim} \Map_\C(- \ast U, C)}, \int_C GC^{\underset{U}{\hocolim} \Map_\C(- \ast U, C)} \right), \]
where the homotopy colimits are taken over all objects $U$ of $\P_0(\mathbf{n+1})$.

To ease some notation in this proof, let us simply denote $\Map_{\Fun(\C,\D)}$ by $\Map$, and let
\[ Z(A,C):= \underset{U}{\hocolim} \Map_\C(A \ast U, C). \]
Thus, we can rewrite our desired map as 
\[ \Map(F,G) \rightarrow \Map\left(\int_C FC^{Z(-,C)}, \int_C GC^{Z(-,C)} \right). \]

To obtain such a map, first observe that the identity morphism $F \rightarrow F$ induces a map
\[ F^{Z(-,C)} \rightarrow F^{Z(-,C)}, \]
which has a corresponding adjoint morphism
\[ Z(-,C) \otimes F^{Z(-,C)} \rightarrow F. \]
Applying the mapping space into $G$, we obtain a map of simplicial sets
\[ \Map(F,G) \rightarrow \Map(Z(-,C) \otimes F^{Z(-,C)}, G), \]
from which we can apply the adjunctions of Definitions \ref{tensor def} and \ref{cotensor def} to get
\[ \Map(F,G) \rightarrow \Map(F^{Z(-,C)}, G^{Z(-,C)}). \]
We can then apply ends over $\C$ to get the map we wanted to define.
\end{proof}

It is a straightforward exercise to show that the analogous results hold for $\bot_n^R$, using the exact same arguments but replacing the fiberwise join $A \ast U$ with the construction $A \sqcup S$ of \eqref{sqcupbotdefn} and reindexing the homotopy colimits by ${\P_0(\mathbf{2})}^{\times n}$ rather than $\mathcal P_0(\mathbf{n+1})$.

\begin{proposition} 
Let $\C$ be a small simplicial category that is tensored in finite simplicial sets and closed  under finite coproducts and $\D$ be a pointed simplicial model category.  If $F \colon \C \rightarrow \D$ is a simplicial functor, then so is $\bot_n^RF$.  Furthermore, the functor
\[ \bot_n^R \colon \Fun(\C,\D) \rightarrow \Fun(\C,\D) \]
given by 
\[ F \mapsto \bot_n^RF \]
is a simplicial functor.
\end{proposition}

\end{document}